\documentclass[reqno, pstricks]{amsart}
\usepackage[T1]{fontenc}
\usepackage[utf8]{inputenc}
\usepackage{bbm}

\usepackage{enumerate,amsmath,amssymb,amscd, caption}
\usepackage{graphicx, soul, tikz}
\usepackage{xcolor, hyperref,fullpage}

\usepackage[all]{xy}
\input xy
\xyoption{all}
\theoremstyle{plain}      
 
\newtheorem{thm}{Theorem}[section]     
     
\newtheorem{cor}[thm]{Corollary}     
     
\newtheorem{lem}[thm]{Lemma}     
\newtheorem{lemma}[thm]{Lemma}
\newtheorem{prop}[thm]{Proposition}

\newtheorem{rem}[thm]{Remark}
\newtheorem{remark}[thm]{Remark}

\theoremstyle{definition}

\newtheorem{definition}[thm]{Definition}     
\newtheorem{example}[thm]{Example}

\renewcommand{\epsilon}{\varepsilon}
\let\theta\vartheta
\let\phi\varphi

\let\na\nabla

\let\<\langle
\let\>\rangle

\DeclareMathAlphabet{\doba}{U}{msb}{m}{n}
\gdef\mC{\doba{C}}
\gdef\mH{\doba{H}}
\gdef\mN{\doba{N}}

\gdef\mR{\doba{R}}
\gdef\mS{\doba{S}}

\gdef\1{\mathbbm{1}}

\def\Spec{{\mathop{\rm Spec}}}
\def\End{{\mathop{\rm End}}}

\def\Spin{{\mathop{\rm Spin}}}
\def\SO{{\mathop{\rm SO}}}
\def\supp{{\mathop{\rm supp}}}
\def\vol{{\mathop{\rm vol}}}
\def\Id{\operatorname{Id}}
\def\dom{{\mathop{\rm dom}}}
\def\ran{{\mathop{\rm ran}}}
\newcommand{\PSO}{P_{\rm SO}}
\newcommand{\PSpin}{P_{\rm Spin}}

\def\Hvec{{\vec{H}}}
\def\trace{{\mathop{\mathrm{tr}}}}
\def\intr{\mathrm{int}}
\def\II{{\rm I\kern-1pt I}}
\let\pa\partial

\let\witi\widetilde
\let\wihat\widehat
\let\Si \Sigma

\let\al\alpha

\let\la\lambda

\let\spec\Spec
\let\C\mC
\let\R\mR

\let\al\alpha

\let\ep\epsilon
\let\om\omega

\newcommand{\dist}{\mathrm{dist}}
\newcommand{\vo}{\mathrm{dvol}}
\renewcommand{\d}{\mathrm{d}}
\renewcommand{\Mc}{\mathbb{M}_c}
\renewcommand{\i}{\mathrm{i}}
\renewcommand{\Im}{\mathrm{Im\,}}%{\mathop{\mathit{Im}}}
\renewcommand{\Re}{\mathrm{Re\,}}%{\mathop{\mathit{Re}}}

\newcommand{\Hom}{\mathop{\mathrm{Hom}}}

\newcommand{\definedas}{\mathrel{\raise.095ex\hbox{\rm :}\mkern-5.2mu=}}
\def\id{\mathop{\rm id}}

\begin{document}
%%%%%%%%%%%%%%%%%%%%%%%%%%%%%%%%%%%%%%%%%%%%%%%%%%%%%%%%%%%%%%%%%
%\printindex

% 
%  \begin{center}
%  \framebox{\framebox{
%  \vbox{This is project {\red \Project}\\
%  Current version {\blue\Version}, from
%  {\blue\Datum}, most recent changes by {\blue\Person}.}}}
%  \end{center}

%\overfullrule 5pt
% {\bf ToDos}
% \begin{enumerate}[(A)]
% \item Einleitung feilen und polieren
% \item Dann: 
% Alles einmal genau gründlich durchgehen (wir beide). Wir haben zuviel 
% kopiert, geändert etc., so dass hier wahrscheinlich noch so manche 
% Kleinigkeit zu tun ist. Das wird zwar leider Zeit kosten, macht den Artikel 
% aber letzendlich viel besser und wertvoller.
% \end{enumerate}
% 
% 
%%%%%%%%%%%%%%%%%%%%%%%%%%%%%%%%%%%%%%%%%%%%%%%%%%%%%%%%%%%%%%%%%

\title 
[$L^p$-spectrum of the Dirac operator]
{$L^p$-spectrum of the Dirac operator on products with hyperbolic spaces}

\author{Bernd Ammann} 
\address{Fakult\"at f\"ur Mathematik, Universit\"at Regensburg, 93040
Regensburg, Germany}
\email{bernd.ammann@mathematik.uni-regensburg.de}

\author{Nadine Gro\ss e} 
\address{Institut f\"ur Mathematik, Universit\"at Leipzig, 04109 Leipzig,
Germany}
\email{grosse@math.uni-leipzig.de}

\subjclass[2010]{ 58J50, 34B27}

\date{\today}

\keywords{Dirac operator, $L^p$-spectrum, Green function, hyperbolic space, 
product spaces}

\begin{abstract}
We study the $L^p$-spectrum of the Dirac operator on complete manifolds. 
One of the main questions in this context is whether this spectrum depends on
$p$. As a first example where $p$-independence fails we compute 
explicitly the $L^p$-spectrum  for the hyperbolic space and its product with compact
spaces. 
\end{abstract}

\maketitle

%%%%%%%%%%%%%%%%%%%%%%%%%%%%%%%%%%%%%%%%%%%%%%%%%%%%%%%%%%%%%%%%%%%%%%%%
\section{Introduction}
%%%%%%%%%%%%%%%%%%%%%%%%%%%%%%%%%%%%%%%%%%%%%%%%%%%%%%%%%%%%%%%%%%%%%%%%

The $L^p$-spectrum of the Laplacian and its $p$-(in)dependence was and still
is studied by many authors, e.g. in  \cite{davies97},
\cite{davies_simon_taylor_88}, \cite{hempel_voigt_86}. On closed manifolds one
easily sees that the spectrum is independent of $p\in [1,\infty]$. For open manifolds, 
independence only holds under additional geometric conditions. 
Hempel and Voigt \cite{hempel_voigt_86},
\cite{hempel_voigt_87} proved such results 
for Schr\"odinger operators in $\mR^n$ with potentials
admitting certain singularities. Then Kordyukov \cite{kord_91} 
generalized this result to  uniformly elliptic operators with uniformly 
bounded smooth coefficients on a manifold of 
bounded geometry with subexponential volume growth. 
Independently,  Sturm \cite{sturm_93} 
showed the independence of 
the $L^p$-spectrum for a class of uniformly elliptic operators 
in divergence form on manifolds with uniformly subexponential volume growth 
and Ricci
curvature bounded from below. Both results include 
the Laplacian acting on functions.  Later the Hodge-Laplacian acting 
on $k$-forms was considered. E.g. under the assumptions of the result by Sturm
from above, Charalambous proved the $L^p$-independence for the 
Hodge-Laplacian in \cite[Proposition 9]{Char05}. The machinery used to obtain
these independence results uses estimates for the 
heat kernel as in \cite{saloff-coste_92}.

In contrast, the $L^p$-spectrum of the Laplacian on
the hyperbolic space does depend on $p$ \cite[Theorem 5.7.1]{davies_90}. 
Its $L^p$-spectrum is the convex hull of 
a parabola 
in the complex plane, and this spectrum degenerates only for $p=2$ to a ray 
on the real axis, cf. Remark~\ref{rem.spec.D2}. 

In addition to the intrinsic interest of the $p$-independence 
of the $L^p$-spectrum, such
results were used to get information on the $L^2$-spectrum by
considering the $L^1$-spectrum, as in particular examples the $L^1$-spectrum
can be easier to control. The result of Sturm
was used for example by Wang \cite[Theorem 3]{wang_97} to prove that the 
spectrum of the Laplacian acting on functions on complete manifolds with
asymptotically non-negative Ricci curvature is $[0,\infty)$.

Explicit calculations for the Laplace-Beltrami operator on locally symmetric spaces were carried out recently by 
Ji and Weber, see e.g. \cite{ji.weber:p10}, \cite{weber:diss}.

About the $L^p$-spectrum of the Dirac operator much less is known. 
As before, on closed manifold the spectrum is independent on $p\in [0,\infty]$.
Kordyukov's methods \cite{kord_91}
do not apply directly to the Dirac operator $D$, but 
following a remark of \cite[Page 224]{kord_91} his methods generalize
to suitable systems, and thus also to the square $D^2$. Unfortunately, the 
system case is not completely worked out, but it seems to us, that the
case of systems is completely analogous to the case of operators on functions. 
Assuming this, Kordyukov has shown that the spectrum of $D^2$ is 
$p$-independent for $1\leq p<\infty$ on manifolds with bounded geometry and
subexponential volume growth. For many such manifolds (e.g. for all such 
manifolds of even dimension or all manifolds of dimension $4k+1$), this 
already implies the $p$-independence of the $L^p$-spectrum of $D$, see
our Lemma~\ref{spec_D2} together with the following symmetry considerations.

Many of the results and techniques that were constructed up 
for Laplace operators are not yet developed 
for Dirac operators. For the Dirac operator such independence results
would not only be of
interest on their own, e.g., for (classical) Dirac operators certain
$L^p$-spaces and $L^p$-spectral gaps naturally occur 
when considering a spinorial Yamabe-type problem which was our  motivation
to enter into this subject, see \cite{ammann.grosse:p13b}.\\

In this paper we determine explicitly the $L^p$-spectrum for a
special class of complete manifolds -- products of compact spaces 
with hyperbolic spaces. More
precisely, we study the following manifolds:  

Let $(N^n,g_N)$ be a closed Riemannian spin manifold. Let $M=\Mc$ be the product
manifold $(\Mc^{m,k}= \mH_c^{k+1}\times N^n,
g_M=g_{\mH_c^{k+1}}+g_N)$ where $\mH_c^{k+1}$ is the $(k+1)$-dimensional
hyperbolic space scaled such that its scalar curvature is $-c^2k(k+1)$ for
$c\neq 0$ and $\mH_0^{k+1}$ is the $(k+1)$-dimensional Euclidean space.
For those manifolds we obtain the following result which is also illustrated in
Figure~\ref{fig_main_thm}:

\begin{thm}\label{main_inv} 
We use the notions from above. Let $p\in [1,\infty]$, and $c\geq 0$.
The $L^p$-spectrum of the Dirac operator on $\Mc^{m,k}=\mH_c^{k+1}\times N^n$ is
given by the set 

\[\sigma_p:= \left\{\mu\in \mC\ \Bigg|\  \mu^2=\lambda_0^2+\kappa^2, |\Im
\kappa|\leq ck\left| \frac{1}{p}-\frac{1}{2}\right| \right\}\]
where $\lambda_0^2$ is the lowest eigenvalue of $(D^N)^2$, $\lambda_0\geq 0$, and
$D^N$ is the Dirac operator on $(N, g_N)$.
In particular, the  Dirac operator $D\colon  H_1^p\to L^p$ on $\Mc^{m,k}$ has a
bounded inverse if and only if
$\lambda_0> ck \left| \frac{1}{p}-\frac{1}{2}\right|$.
\end{thm}

For an overview of the structure of the proof, see the end of the introduction.

From the Theorem \ref{main_inv} one can directly read of the $L^p$-spectrum of
$D^2$ and compare it to the known spectrum of the Laplacian acting on functions
which is done in Remark \ref{rem.spec.D2}.\\

\begin{figure}[h]
 \centering
\begin{tikzpicture}
 \path[draw, ->] (0,0)   to (3,0);
 \path (3.4,0.2) node {\small $\Re \mu$};
 \path[draw, ->]  (1.5,-1.5) to (1.5,1.5);
 \path (1.9, 1.6) node {\small $\Im \mu$};
\fill[fill=gray, opacity=0.3] (0,1)  -- (3,1) -- (3,-1) -- (0,-1) -- cycle;
 \path[draw, line width = 1.2]  (0,1) to (3,1);
 \path[draw, line width = 1.2]  (0,-1) to (3,-1);
\path[draw, ->] (0.5,1.5) node {\small $x_L$} (0.7,1.4) to (1.4, 1.1);

 \path[draw, ->] (4.5,0)   to (7.5,0);
  \path (7.9,0.2) node {\small $\Re \mu$};
  \path[draw, ->]  (6,-1.5) to (6,1.5);
  \path (6.4, 1.6) node {\small $\Im \mu$};
  \path[draw, dotted]  (4.5,1) to (7.5,1);
  \path[draw, dotted]  (4.5,-1) to (7.5,-1);
\fill[fill=gray, opacity=0.3] plot[domain=7.5:4.5,samples=30] (\x,{ 
-0.85*(1.2-1/(1+3*(\x-6)^2))  }) -- plot[domain=4.5:7.5,samples=30] (\x,{ 
0.85*(1.2-1/(1+3*(\x-6)^2))  }) --  cycle;
  \draw[line width=1] plot[domain=4.5:7.5,samples=50] (\x,{ 
0.85*(1.2-1/(1+3*(\x-6)^2))  });
  \draw[line width=1] plot[domain=4.5:7.5,samples=50] (\x,{ 
-0.85*(1.2-1/(1+3*(\x-6)^2))  });
  \path[draw, ->] (9,0)   to (12,0);
  \path (12.4,0.2) node {\small $\Re \mu$};
  \path[draw, ->]  (10.5,-1.5) to (10.5,1.5);
  \path (10.9, 1.6) node {\small $\Im \mu$};
\path[draw, ->] (5,1.5) node {\small $x_M$} (5.05,1.35) to (5.95, 0.25); 

  \path[draw, dotted]  (9,1) to (12,1);
  \path[draw, dotted]  (9,-1) to (12,-1);
\fill[fill=gray, opacity=0.3] (12,0.9) -- plot[domain=11.5:10.8,samples=30]
(\x,{ 0.5*sqrt(3*(\x-10.8))}) -- plot[domain=10.8:11.5,samples=30] (\x,{
-0.5*sqrt(3*(\x-10.8))}) --  (12,-0.9) -- cycle ;
  \draw[line width=1] plot[domain=10.8:11.5,samples=30] (\x,{
0.5*sqrt(3*(\x-10.8))});
 \path[draw, line width = 1] (11.5,0.725)  to (12,0.9) ;
 \draw[line width=1] plot[domain=10.8:11.5,samples=30] (\x,{
-0.5*sqrt(3*(\x-10.8))});
 \path[draw, line width = 1] (11.5,-0.725)  to (12,-0.9) ;
 \fill[fill=gray, opacity=0.3] (9,0.9) -- plot[domain=9.5:10.2,samples=30] (\x,{
0.5*sqrt(-3*(\x-10.2))}) -- plot[domain=10.2:9.5,samples=30] (\x,{
-0.5*sqrt(-3*(\x-10.2))}) --  (9,-0.9) -- cycle ;
 \draw[line width=1] plot[domain=9.5:10.2,samples=30] (\x,{
0.5*sqrt(-3*(\x-10.2))});
 \path[draw, line width = 1] (9.5, 0.725)  to (9,0.9) ;
 \draw[line width=1] plot[domain=9.5:10.2,samples=30] (\x,{
-0.5*sqrt(-3*(\x-10.2))});
 \path[draw, line width = 1] (9.5,-0.725)  to (9,-0.9) ;
\path[draw, ->] (9.5,1.5) node {\small $x_R$} (9.5,1.35) to (10.75, 0.1); 

  \end{tikzpicture}
 \caption{The shaded region (including the boundary) illustrates the
$L^p$-spectrum of the Dirac operator on $\Mc^{m,k}=\mH_c^{k+1}\times N^n$, cf.
Theorem~\ref{main_inv}.\\ 
Left: $\lambda_0=0$ $\left(x_L=
ck\left|\frac{1}{p}-\frac{1}{2}\right|\right)$.\\
 Middle: $0<\lambda_0<ck\left|\frac{1}{p}-\frac{1}{2}\right|$ $\left(x_M=\left(
c^2k^2\left(\frac{1}{p}-\frac{1}{2}\right)^2-\lambda_0^2\right)^{\frac{1}{2}}
\right)$. \\
 Right: $\lambda_0>ck\left|\frac{1}{p}-\frac{1}{2}\right|$ $\left(x_R=\left(
\lambda_0^2-c^2k^2\left(\frac{1}{p}-\frac{1}{2}\right)^2\right)^{\frac{1}{2}}
\right)$.
}
\label{fig_main_thm}
\end{figure}

The paper is structured as follows: Notations and preliminaries are collected in
Section~\ref{sec.prelim}. Results on the Green function of the Dirac operator
acting on $L^2$-spinors can be found in Section~\ref{sec.green}. General remarks and
results for the Dirac operator acting on $L^p$-sections are given in
Appendix~\ref{sec_Lp}.

In Section~\ref{sec.D.hyp}, the Dirac operator on the model spaces $\Mc^{m,k}$
is written in polar coordinates and the action of $\Spin(k+1)$ on $\Mc^{m,k}$ is
studied. This is used in Section~\ref{sec_symm} to prove a certain symmetry
property of the Green function on $\Mc^{m,k}$ and in Section~\ref{sec.decay} to
study its decay.

After all these preparations we are ready to prove the main theorem:\\

{\it Structure of the proof of Theorem~\ref{main_inv}}

Section~\ref{sec_Green}: We decompose the Green function into a singular part
and a smoothing operator. Using the homogeneity of the hyperbolic space we show
in Proposition~\ref{Green_sing} that the singular part gives rise to a bounded
operator from $L^p$ to itself for all $p\in[1,\infty]$. In
Proposition~\ref{Green_smooth} we show that under certain assumptions on the
decay of the Green function also the smoothing part gives rise to a bounded
operator from $L^p$ to $L^p$ for certain $p$.

Section~\ref{sec.first}: Using the decay estimate obtained in
Section~\ref{sec.decay} we then see that the $L^p$-spectrum of $\Mc^{m,k}$ is
contained in the set $\sigma_p$ given in Theorem~\ref{main_inv}.

Thus, it only remains to show that each element of $\sigma_p$ is already in the
$L^p$-spectrum of $\Mc^{m,k}$. For that we construct test spinors on
$\mH_c^{k+1}$ in Section~\ref{sec.const} and finish the proof for product spaces
in Section~\ref{sec.last}.

%%%%%%%%%%%%%%%%%%%%%%%%%%%%%%%%%%%%%%%%%%%%%%%%%%%%%%%%%%%%%%%%%%%%%
%
% Preliminaries 
%
%%%%%%%%%%%%%%%%%%%%%%%%%%%%%%%%%%%%%%%%%%%%%%%%%%%%%%%%%%%%%%%%%%%%%

%%%%%%%%%%%%%%%%%%%%%%%%%%%%%%#
\section{Preliminaries}\label{sec.prelim}
%%%%%%%%%%%%%%%%%%%%%%%%%%%%%%%%%%

\subsection{Notations and conventions} In the article we will use the convention
that a spin manifold is a manifold which admits a spin structure together with a
fixed choice of spin structure.

Let $(M,g)$ be a spin manifold and $\Si_M$ the corresponding spinor bundle, see
Section~\ref{spin_prelim}.\\

$\Gamma (\Si_M)$ denotes the space of spinors, i.e., sections of $\Si_M$. The
space of smooth compactly supported sections is denoted by $C_c^\infty(M,
\Si_M)$, or shortly $C_c^\infty(\Si_M)$. The hermitian metric on fibers of $\Si_M$
is denoted by $\<.,.\>$, the corresponding norm by $|.|$. For $s_1,s_2\in
\Gamma(M, \Si_M)$ we define the $L^2$-scalar product
  \[(s_1,s_2)_{L^2(g)}:= \int_M \<s_1,s_2\> \,\vo_g.\]

For $s\in [1,\infty]$ $\Vert . \Vert_{L^s(g)}$ is the $L^s$-norm on $(M^n,g)$.
In case the underlying metric is clear from the context we abbreviate shortly by
$\Vert. \Vert_s$.

$\Spec_{L^s}^M(D)$ denotes the spectrum of the Dirac operator on $M$ viewed as
an operator from $L^s$ to $L^s$, cf. Appendix~\ref{sec_Lp}. 

 We denote by $\pi_i\colon  M\times M\to M$, $i=1,2$, the
projection to the $i$-th component. Moreover, we set $\Si_M  \boxtimes
\Si_M^*:=\pi_1^*(\Si_M)\otimes (\pi_2^*(\Si_M^*))$.

% $H^{q, \nabla}_k$ denotes both the space of distributions on $M$ and the one of
% distributional sections in $\Si_M$ that have finite $H_k^{q, \nabla}$-norm given
% by 
% \[ \Vert \phi\Vert_{H_k^{q, \nabla}}^q=\sum_{i=0}^k \Vert (\nabla)^i
% \phi\Vert_{L^q}^q.\] Here $\nabla$ denotes the covariant derivative on $M$ and
% $\Si_M$, respectively, dependent whether $\phi$ is a distribution on $M$ and a
% distributional section in $\Si_M$, respectively. 
% $H_{k,loc}^{q, \nabla}$ means that any restriction of the distribution to a
% relatively compact subset has to be in $H^{q, \nabla}_k$ of that subset.

$C^i(M)$ denotes the space of $i$-times continuously differentiable functions on
$M$.

 $B_\epsilon(x)\subset M$ is the ball around $x\in M$ of radius $\epsilon$
w.r.t. the metric given on $M$.
 
A Riemannian manifold is of bounded geometry, if its injectivity
radius is positive and the curvature tensor and all derivatives are bounded.

The metric on the $k$-dimensional sphere $\mS^k$ with constant sectional
curvature $1$ will be denoted by $\sigma^k$. For $\mS^k$ with metric
$r^2\sigma^k$ we write $\mS^k_r$.

\subsection{Coordinates and notations for $\mH_c^{k+1}$ and its product
spaces.}\label{sec_H}

We introduce coordinates on $\mH_c^{k+1}$ by  equipping $\mR^{k+1}$
with the metric $g_{\mH_c^{k+1}}=\d r^2 +f(r)^2\sigma^k $ where $\sigma^k$ is
the standard metric on $\mS^k$ and
\[ f(r):= \sinh_c(r):= \left\{ \begin{matrix} \frac{1}{c}\sinh(cr) &\quad
\mathrm{if\ }c\ne 0\ \\
r&\quad \mathrm{if\ }c=0. \end{matrix} \right. \]
In particular, the distance $\dist_{\mH_c^{k+1}}$ of $y$ to $0$ w.r.t.
$g_{\mH_c^{k+1}}$ coincides with the euclidean one on $\mR^{k+1}$. The subset
$\{ y\in \mH_c^{k+1}\ |\ \dist_{\mH_c^{k+1}}(y,0)=r\}$ is isometric to
$\mS_{f(r)}^k$ and its (unnormalized) mean curvature is given by

\begin{equation*}
\Hvec_{\mS^k_{f(r)}}= -k \frac{\partial_r
f(r)}{f(r)}\pa_r=-k\coth_c(r)\partial_r\quad \text{where\ } \coth_c{r}:=\left\{ 
\begin{matrix} c\coth(cr)  &\quad\mathrm{if\ }c\ne 0\ \\
\frac{1}{r}&\quad\mathrm{if\ }c=0. 
\end{matrix} \right. 
\end{equation*}

The identity induces a map $\mR^{k+1}\to \mH_c^{k+1}$. Unless otherwise stated
we use this map to identify $\mH_c^{k+1}$ with $\mR^{k+1}$ as a manifold.

Let $N$ be a closed Riemannian spin manifold. Note that we include the case
where $N$ is just a point. Set $\Mc^{m,k}:= \mH_c^{k+1}\times N$, and
$\pi_{\mH}$ shall denote the projection of $\Mc^{m,k}$ onto its
$\mH_c^{k+1}$-coordinates.

\subsection{General preliminaries about spin geometry}\label{spin_prelim}
The following can e.g. be found in \cite{friedrich:00}.
A spin structure on $M^m$ is a pair $(P_\Spin(M),\alpha)$ where $P_\Spin(M)$ is
a principal $\Spin(m)$-bundle and where $\alpha\colon P_\Spin(M)\to P_\SO(M)$ is a
fiber map over the identity of $M$ that is compatible with the double covering
$\Theta\colon  \Spin(m)\to \SO(m)$ and the corresponding group actions, i.e.,  the
following diagram commutes

\begin{equation*}
\xymatrix{
\Spin(m)\times P_\Spin (M)\ar[dd]^{\Theta\times \alpha} \ar[r] & P_\Spin
(M)\ar[dd]^{\alpha} \ar[dr] & \\
&&M\\
\SO(m)\times P_\SO (M) \ar[r] &P_\SO (M)\ar[ur]&
}
\end{equation*}

Let $\Si_m$ be an irreducible representation of ${\rm Cl}_m$. In case $m$ is odd there are two such irreducible representations. Both of them coincide if considered as $\Spin(m)$-representations. If $m$ is even, there is only one irreducible ${\rm Cl}_m$-representation of $\Si_m$, but it splits into non-equivalent subrepresentations $\Si_m^{(+)}$ and $\Si_m^{(-)}$ as $\Spin(m)$-representations.

Let $\ep\in \{+, -\}$. We use the notation $\Si_{m}^{(\ep)}$ if $m$ is odd as well and set in this case $\Si_{m}^{(\ep)}=\Si_m$.

The spinor bundle $\Sigma_M$ is defined as $\Sigma_M=P_\Spin
(M)\times_{\rho_m} \Sigma_m$ where $\rho_m\colon  \Spin(m)\to {\rm End}(\Sigma_m)$ is
the complex spinor representation. Moreover, the spinor bundle is endowed with a
Clifford multiplication, denoted by '$\cdot$', $\cdot\colon  TM\to {\rm End}(\Si_M)$. 
Then, the Dirac operator acting on the space of smooth sections of $\Si_M$ is
defined as the composition of the connection $\nabla$  on $\Si_M$ (obtained as a
lift of the Levi-Civita connection on $TM$) and the Clifford multiplication.
Thus, in local coordinates this reads as
\[ D=\sum_{i=1}^m e_i\cdot \nabla_{e_i}\]
where $(e_i)_{i=1,\ldots,m}$ is a local orthonormal basis of $TM$. The Dirac
operator is  formally self-adjoint as an operator on $L^2$, i.e., for $\psi\in C^\infty(M, \Si_M)$ and $\phi\in C_c^\infty(M, \Si_M)$
we have $(\phi, D\psi)=(D\phi, \psi)$.

As $M$ is complete, the Dirac operator is not only formally self-adjoint, but actually has a self-adjoint extension that is a densely defined operator $D\colon  L^2\to L^2$, see \cite{Wolf}. From the spectral theorem it then follows that $D-\mu\colon  L^2\to L^2$ is invertible for
all $\mu\not\in \mR$. 

Define $\omega_M=\i^{[\frac{m+1}{2}]} e_1\cdot e_2\cdot \ldots \cdot e_m$ with $(e_i)_i$ being a positively oriented orthonormal frame on $M$.
If $m$ is even, $\omega_M^2=1$ 
and the corresponding $\pm 1$ eigenspaces are the spaces of so-called positive
(resp. negative) spinors.

\subsection{Dual spinors}

The hermitian metric induces a natural isomorphism from $\Si_M^*$ to $\bar{\Si}_M$. In this way we obtain a metric connection and a Clifford multiplication on $\Si_M^*$ and this allows us to define a Dirac operator $D^{\rm t}\colon  C^\infty(\Si_M^*)\to  C^\infty(\Si_M^*)$. Locally $D^{\rm t}f=\sum_i e_i\cdot \nabla_{e_i} f$ where $f\in C^\infty(\Si_M^*)$ and $e_i$ is a local orthonormal frame on $M$. Completely analogously to the proof that the usual Dirac operator is formally self-adjoint, one proves that for $f\in C^\infty(\Si_M^*)$, $\phi\in C^\infty(\Si_M)$ such that $\supp\, f\cap \supp\, \phi$ is relatively compact we have
\[ \int D^{\rm t}f(\phi)\vo_g=\int f(D\phi)\vo_g.\]

\subsection{ Spinors on product manifolds}\label{spin_prod} In this subsection our notation is close to \cite{baer_98}.
Let
$(P^{m+n}=M^m\times N^n$,
$g_P=g_M +g_N)$ be a product of Riemannian spin manifolds $(M,g_M)$ and
$(N,g_N)$.
We have 
  $$P_\Spin(M\times N)=(P_\Spin(M) \times P_\Spin( N))\times_\xi \Si_{m+n}$$
where $\xi\colon\Spin(m)\times \Spin(n)\to \Spin(m+n)$ is the Lie group
homomorphism lifting the 
standard embedding $\SO(m)\times \SO(n)\to \SO(m+n)$. Note that~$\xi$ is not an
embedding, its kernel is
$(-1,-1)$, where~$-1$ denotes the non-trivial element in the kernel of
$\Spin(m)\to \SO(m)$ resp. $\Spin(n)\to \SO(n)$.

The spinor bundle can be identified with

\[ \Si_P= \left\{\begin{matrix} \Si_M \otimes (\Si_N\oplus \Si_N) & \text{if both $m$ and $n$ are odd}\\ \Si_M \otimes \Si_N & \text{else},
 \end{matrix} \right.\]

and the Levi-Civita connection  acts as $\nabla^{\Si_M\otimes \Si_N}=
\nabla^{\Si_M}\otimes \Id_{\Si_N} + \Id_{\Si_M}\otimes \nabla^{\Si_N}$. This identification can be chosen such that for
$X\in TM$, $Y\in TN$, $\phi\in \Gamma(\Si_M)$, and $\psi=(\psi_1, \psi_2)\in \Si_N\oplus\Si_N$ for both $n$ and $m$  odd and $\psi\in \Gamma(\Si_N)$ otherwise, we have

\begin{equation*}%\label{Cliff-prod} 
 (X,Y)\cdot_P (\phi\otimes \psi)= (X\cdot_M
\phi) \otimes (\omega_N\cdot_N \psi) + \phi\otimes (Y\cdot_N \psi)\end{equation*}
where for both $n$ and $m$ odd we set  $\omega_N\cdot_N (\psi_1, \psi_2):= \i (\psi_2, -\psi_1)$ and  $Y\cdot_N (\psi_1, \psi_2):= (Y\cdot_N \psi_2, Y\cdot_N \psi_1)$.

The Dirac operator is then given by
\begin{equation*}%\label{Dir-prod} 
 D^P(\phi\otimes \psi)= (D^M \phi \otimes
\omega_N\cdot_N \psi) + (\phi\otimes \tilde{D}^N \psi)\end{equation*}
where $\tilde{D}^N={\rm diag} (D^N, -D^N)$ if both $m$ and $n$ are odd and
$\tilde{D}^N=D^N$ otherwise. 

Since $\omega_N\cdot$ and $\tilde{D}^N$ anticommute, $D^M\otimes \om_N$ and
$\id\otimes \tilde{D}^N$ anticommute as well. Thus 
  \begin{align}\label{Dirac_square_prod}  (D^P)^2= (D^M)^2\otimes \id + \id
\otimes (\tilde{D}^N)^2.
   \end{align}

\subsection{A covering lemma.}

\begin{lemma}[Covering lemma]\label{cover} Let $(M,g)$ be a Riemannian manifold
of bounded geometry, and let $R>0$. Then there are points $(x_i)_{i\in I}\subset
M$ where $I$ is a countable index set such that 

\begin{itemize}
 \item[(i)] the balls $B_R(x_i)$ are pairwise disjoint and
 \item[(ii)] $(B_{2R}(x_i))_{i\in I}$ and $(B_{3R}(x_i))_{i\in I}$ are both
uniformly locally finite covers of $M$.
\end{itemize}
\end{lemma}

\begin{proof}
Choose a maximal family of points $(x_i)_{i\in I}$ in  $M$ such that 
the sets
$B_{R}(x_i)$ are pairwise disjoint. Then  $\bigcup_{i\in I} B_{2R}(x_i)=M$.  
For $y\in M$ let $L(y)= \{ i\in I \ |\ y\in 
B_{3R}(x_i)\}$. For $i\in L(y)$ we have
$B_{R}(x_i)\subset B_{4R}(y)$ and, thus, 
\[\bigsqcup_{i\in L(y)} B_{R}(x_i)\subset  B_{4R}(y),\]
where $\sqcup$ denotes disjoint union.
Comparing the volumes of both sides and using the bounded geometry of $M$ we see
that
there exists a number $L_R$ such that $|L(y)|\leq L_R$ for all $y\in M$. Thus,
the
covering by sets $B_{3R}(x_i)$, and hence the one by $B_{2R}(x_i)$, is uniformly
locally finite. 
\end{proof}

 \subsection{ Interpolation theorems}

\begin{thm}[Riesz-Thorin Interpolation Theorem, { \cite[Theorem II.4.2]{werner_00}}] \label{RTint}
Let $T$ be an operator defined on a domain $\mathcal{D}$ that is dense in both
$L^q$ and $L^{p}$. Assume that $Tf\in L^q\cap L^p$ for all $f\in \mathcal{D}$
and that $T$ is bounded in both norms. Then, for any $r$ between $p$ and $q$ the
operator $T$ is a bounded operator from $L^r$ to $L^r$.
\end{thm}

\begin{thm}[Stein Interpolation Theorem, {\cite[Section 1.1.6]{davies_90},
\cite[Theorem IX.21]{RS2}}]\label{stein_int}
 Let $p_0,q_0,p_1,q_1\in[1,\infty]$, $0<t<1$, and $S=\{z\in \C\ |\ 0\leq \Re
z\leq 1\}$.
Let $A_z$ be linear operators from $L^{p_0}\cap L^{p_1}$ to $L^{q_0}+L^{q_1}$
for all $z\in S$ with the following properties
\begin{itemize}
 \item[(i)] $z\mapsto \< A_zf, g\>$ is uniformly bounded and continuous on $S$ and
analytic in the interior of $S$ whenever $f\in L^{p_0}\cap L^{p_1}$ and  $g\in
L^{r_0}\cap L^{r_1}$ where $r_i$ is the conjugate exponent to $q_i$.
\item[(ii)] There is $M_0>0$ such that $\Vert A_{\i y}f\Vert_{q_0}\leq M_0\Vert
f\Vert_{p_0}$ for all $f\in L^{p_0}\cap L^{p_1}$ and $y\in \R$.
 \item[(iii)] There is $M_1>0$ such that $\Vert A_{1+\i y}f\Vert_{q_1}\leq
M_1\Vert f\Vert_{p_1}$ for all $f\in L^{p_0}\cap L^{p_1}$ and $y\in \R$.
\end{itemize}
Then, for $1/p=t/p_1+(1-t)/p_0$ and  $1/q=t/q_1+(1-t)/q_0$
\[ \Vert A_tf\Vert_q\leq M_1^tM_0^{1-t}\Vert f\Vert_p\]
for all $f\in L^{p_0}\cap L^{p_1}$. Hence, $A_t$ can be extended to a bounded
operator from $L^p$ to $L^q$ with norm at most $M_1^tM_0^{1-t}$.
\end{thm}

\section{The Green function}\label{sec.green}

In this section, we collect results on existence and properties of the Green
function of the Dirac operator $D$ and its shifts $D-\mu$, $\mu\in \mC$. They are obvious applications of standard
methods, but a suitable reference does not exist yet. Unless
otherwise stated we only assume in this section that the Riemannian spin
manifold $(M,g)$ is complete.

\begin{definition}\cite[Definition~2.1]{AHM} A smooth section $G_{D-\mu}\colon M\times
M \setminus \Delta\to \Si_M\boxtimes \Si_M^*$  that is locally integrable on
$M\times M$
is called a \emph{Green function} of the shifted Dirac operator $D-\mu$    
if 
\begin{equation}(D_x-\mu)(G_{D-\mu}(x,y)) = \delta_y
\Id_{\Si_M|_y}\label{def.green}\end{equation}
in the sense of distributions, i.e.,        
for any $y\in M$, $\psi_0\in \Si_M|_y $, and     
$\phi\in C_c^\infty(\Si_M)$
\[\int_M \<G_{D-\mu}(x,y)\psi_0,(D-\bar{\mu})\phi(x)\>\d x= \<\psi_0,\phi(y)
\>\]
and $G_{D-\mu}(., y)\in L^2(M\setminus B_r(y))$ for any $r>0$.
\end{definition}

In case that the operator $D-\mu$ is clear from the context, we shortly write
$G=G_{D-\mu}$.

\begin{prop}\label{Green-compact}If $M$ is a closed Riemannian spin manifold
with invertible  operator $D-\mu\colon L^2(\Si_M)\to L^2(\Si_M)$, then a unique Green
function exists.
\end{prop}

To prove the well-known proposition, one usually starts by showing the existence of a 
parametrix. 

\begin{lemma}\cite[III.\S 4]{lawson.michelsohn:89} Let $M$ be a closed Riemannian spin manifold. Then there is a
smooth section
$P_{D-\mu}\colon M\times M \setminus \Delta\to \Si_M\boxtimes \Si_M^*$, called parametrix,  which is $L^1$
on 
$M\times M$ and which satisfies  
  \[(D_x-\mu)(P_{D-\mu}(x,y)) = \delta_y Id_{\Si_M|_y} + R(x,y)\]
  in the distributional sense  
for a smooth section $R\colon M\times M \to \Si_M\boxtimes \Si_M^*$.
\end{lemma}

Convolution with $P_{D-\mu}$ thus defines an operator $\mathtt{P}_{D-\mu}$ by
\[(\mathtt{P}_{D-\mu}\psi,\phi)=\int_M\int_M \<P_{D-\mu}(x,y) \psi(y), \phi(x)\>
\d x\d y\]  for all $\psi,\phi\in C_c^\infty(\Si_M)$. Then, $\mathtt{P}_{D-\mu}$
is a right inverse
to ${D-\mu}$ up to infinitely smoothing operators. We thus call it a right
parametrix.
The existence of such a right parametrix follows using the symbol calculus
from the fact that $D$ is an elliptic operator. 
An efficient and very readable overview  over how to construct a right 
parametrix for an elliptic differential operator on a compact manifold
can be found e.g.\ in \cite[III.\S 4]{lawson.michelsohn:89}, although the 
reader should pay attention to the fact that it is not so obvious that
the different notions of infinitely smoothing operators used in there 
are in fact all equivalent. The latter fact follows from standard 
techniques used in the theory of pseudo differential operators, see e.g.\ 
\cite{abels:12} or \cite{taylor:81} for textbooks on this subject.
% The construction in \cite[III.\S 4]{lawson.michelsohn:89} even provides
% a $P_{D-\mu}$ such that convolution with $P_{D-\mu}$ is actually a right and
% left inverse
% of ${D-\mu}$ up to infinitely smoothing operators. 

\begin{proof}[Proof of Proposition \ref{Green-compact}]
From the last Lemma we have the existence of a parametrix $P_{D-\mu}(x,y)$. We
will use the notations of that Lemma. Since $D-\mu$ is assumed to be invertible, there
is a section $P_{D-\mu}'\colon M\times M\to \Si_M\boxtimes \Si_M^*$ with $(D_x-\mu)
P'_{D-\mu}(x,y)=R(x,y)$. By elliptic regularity $P'_{D-\mu}$ is smooth in $x$
and $y$. We set $G_{D-\mu}(x,y)=P_{D-\mu}(x,y)-P'_{D-\mu}(x,y)$ and obtain
$(D_x-\mu)(G_{D-\mu}(x,y))=\delta_y \Id_{\Si_M|_y}$. Moreover, since $P_{D-\mu}$
is $L^1$ on $M\times M$ and $P_{D-\mu}'$ is smooth in both entries the Green
function $G_{D-\mu}$ is $L^1$ as well. 
Furthermore, $P_{D-\mu}(.,y)$ is smooth on $M\setminus B_r(y))$ for any $r>0$ and,
hence, the same is true for $G_{D-\mu}(.,y)$. In particular, $G_{D-\mu}(.,y)\in
L^2(M\setminus B_r(y))$. If $\tilde{G}_{D-\mu}$ is a possibly different Green function of $D-\mu$
then $(D-\mu)({G}_{D-\mu}(.,y)- \tilde{G}_{D-\mu}(.,y))=0$ for all $y\in M$. As $D-\mu$ is invertible
we have ${G}_{D-\mu}=\tilde{G}_{D-\mu}$.
\end{proof}

As for $P_{D-\mu}$, convolution with $G_{D-\mu}$ defines an operator
$\mathtt{G}_{D-\mu}$ by
\[(\mathtt{G}_{D-\mu}\psi,\phi)=\int_M\int_M \<G_{D-\mu}(x,y) \psi(y), \phi(x)\>
\d x\d y\]  for all $\psi,\phi\in C_c^\infty(\Si_M)$. By construction
$\mathtt{G}_{D-\mu}$ is the right inverse of $D-\mu$, and is thus even defined
on $L^2$. Since the inverse of $D-\mu$ exists by assumption,
$\mathtt{G}_{D-\mu}=(D-\mu)^{-1}$, and  $\mathtt{G}_{D-\mu}$ is in particular
also a left inverse of $D-\mu$.

\begin{lemma}\label{adj.green} Let $M$ be a closed Riemannian spin manifold, and let $D-\mu$ be invertible. Then
$G_{D-\mu}(x,y)$ is the adjoint of $G_{D-\bar{\mu}}(y,x)$, i.e.
$G_{D-\bar{\mu}}(y,x)$ is the integral kernel of the adjoint operator of
$\mathtt{G}_{D-\mu}$.
\end{lemma}

\begin{proof} Using the definitions and discussions from above and Lemma~\ref{conj_spectrum}(ii) we have
$\mathtt{G}_{D-\mu}^*=((D-\mu)^{-1})^*=(D-\bar{\mu})^{-1}=\mathtt{G}_{D-\bar{\mu
}}$. In particular, we get for all $\psi, \phi\in L^2(\Si_M)$ that
\begin{align*}
(\psi,\mathtt{G}_{D-\mu}^*\phi)=& (\mathtt{G}_{D-\mu}\psi,\phi)=
((D-\mu)^{-1}\psi,\phi) = (\psi, (D-\bar{\mu})^{-1}\phi)\\
&=\int\int \< \psi(y), G_{D-\bar{\mu}}(y,x)\phi(x)\> \d y\d x.
\end{align*}
\end{proof}

Moreover, we have

\begin{lemma}\label{compact_sym_Green}
In the situation of Lemma~\ref{adj.green}  we have $(D_y^{\rm t}-\mu) G_{D-\mu}(x,y)=\delta_x\Id_{\Si_M^*|_x}$, i.e., for $f_0\in \Gamma ({\Si_M^*}|_{x})$, $\phi\in C_c^\infty
(\Si_M)$ 
\begin{align*}
 \int ((D_y^{\rm t}-\mu) G_{D-\mu}(x,y) f_0)(\phi(y))\d y= f_0(\phi(x)).
\end{align*}

\end{lemma}

\begin{proof} 

\begin{align*}
 \int ((D_y^{\rm t}-\mu) G_{D-\mu}(x,y) f_0)(\phi(y)) \d y=& \int (G_{D-\mu}(x,y) f_0)( (D_y-\mu) \phi(y))\d y\\
 =& \int f_0 (G_{D-\mu}(x,y) (D_y-\mu) \phi(y))\d y\\
 =& f_0(\phi(x)).
\end{align*}
where the last step uses that $\mathtt{G}_{D-\mu}$ is also the left inverse of $D-\mu$. 
% the claim follows immediately from the last Lemma when viewing
% $(D_y-\mu)G_{D-\mu}(x,y)$ as an operator from $\Gamma(\Si^*)$ to itself:
% \begin{align*}
%  \int \< G_{D-\mu} (x,y) (D_y-\mu)\psi(y), \phi(x)\>\d y=& \int \< 
% (D_y-\mu)\psi(y), G_{D-\bar{\mu}} (y,x) \phi(x)\>\d y\\
% =&\< \psi(x), \phi(x)\>.
% \end{align*}

\end{proof}

Now, $M$ has no longer to be closed, but we assume bounded geometry.

\begin{prop}\label{prop.green.exists} 
Let $(M,g)$ be a Riemannian spin manifold of bounded geometry. Let $D-\mu\colon
L^2(\Si_M)\to L^2(\Si_M)$ be invertible. Then there exists a unique Green function.
\end{prop}

\begin{proof}
We choose $R>0$ such that  $3R$ is smaller than the injectivity radius. 
Let  $(x_i)_{i\in I}$ be as in the Covering Lemma~\ref{cover}.
We define $(M\times M)_\ep:=\{(x,y)\in M\times M\,|\, \dist(x,y)<\ep\}$. 
Because of $M=\bigcup_{i\in I} B_{2R}(x_i)$ we have
\[(M\times M)_R \subset\bigcup_{i\in I} B_{3R}(x_i)\times B_{3R}(x_i).\] 

We embed 
each ball $B_{3R}(x_i)$  isometrically into a closed connected 
manifold $M_{x_i}$, which is diffeomorphic to a sphere and $D^{M_{x_i}}-\mu$ is
invertible. This can always be achieved by local metric
deformation on $M_{x_i}\setminus B_{3R}(x_i)$, see Proposition~\ref{app_C}. 

Thus, by Proposition \ref{Green-compact} the operator $D^{M_{x_i}}-\mu$
possesses a Green function  
$G^{x_i}(x,y)$ with $(D^{M_{x_i}}_x-\mu)G^{x_i}(x,y)=\delta_y \Id_{\Si_y}$.
By abuse of notation we will view $G^{x_i}(x,y)$ for $x,y\in B_{3R}(x_i)$ also
as a partially defined section of 
$\Si_M\boxtimes \Si_M^*\to M\times M$, which is defined on $B_{3R}(x_i)\times
B_{3R}(x_i)$.

Now we choose smooth functions $a_i$ on $M\times M$
such that $\supp\, a_i\subset B_{3R}(x_i)\times  B_{3R}(x_i)\subset (M\times
M)_{6R}$ and such that 
$\sum_{i\in I}a_i$ equals to $1$ on $(M\times M)_{R/2}$.
Now we set
\begin{align*}
 H(x,y)=\sum_{i\in I} a_i(x,y)G^{x_i}(x,y). 
\end{align*}
This implies $\supp\, H\subset (M\times M)_{6R}$. Moreover, $H(.,y)\in
L^2(M\setminus B_r(y))$ for all $r>0$ since this is true for each summand.

Our next goal is to prove that $(D_x-\mu) H(x,y) - \delta_y \Id_{\Si_y}$ is
smooth.
Note that $G^{x_i}(x,y)$
and $G^{x_j}(x,y)$ are both defined for $(x,y)\in (B_{3R}(x_i)\times  B_{3R}(x_i))\cap  (B_{3R}(x_j)\times
 B_{3R}(x_j))$, but they will not coincide in general. On
the other hand their defining property and the 
locality of the differential operator $D$ (cp. Lemma~\ref{compact_sym_Green})
imply that 
  \[(D_x-\mu)\left(G^{x_i}(x,y)-G^{x_j}(x,y)\right)=
(D_y^{\rm t}-\mu)\left(G^{x_i}(x,y)-G^{x_j}(x,y)\right)=0.\]
Thus, 
\[ \underbrace{((D_x-\mu)^2+(D_y^{\rm t}-\mu)^2)}_{=:P} \left(G^{x_i}(x,y)-G^{x_j}(x,y)\right)=0.\]
Since $P$ is an elliptic operator, elliptic regularity implies that $G^{x_i}(x,y)-G^{x_j}(x,y)$ viewed as a difference of distributions is a smooth function on
$(B_{3R}(x_i)\times  B_{3R}(x_i))\cap  (B_{3R}(x_j)\times  B_{3R}(x_j))$, and
thus
$a_j(x,y)\left(G^{x_i}(x,y)-G^{x_j}(x,y)\right)$ as well. 
On $B_{3R}(x_j)\times  B_{3R}(x_j)$ we rewrite
  \[H(x,y)= G^{x_j}(x,y) + \sum_{i\in I\setminus\{j\}}
a_i(x,y)\left(G^{x_i}(x,y)-G^{x_j}(x,y)\right),\]
and we conclude that 
$(D_x-\mu)H(x,y)=\delta_y \Id_{\Si_y} + F(x,y)$ where $F(x,y)$ is a smooth
section of $\Si_M\boxtimes \Si_M^*$ with support in $(M\times M)_{6R}$.

There is a unique section $H'$ of $\Si_M\boxtimes \Si_M^*$ such that
$(D_x-\mu)H'(x,y)=F(x,y)$ and such that $H'(.,y)$ is
$L^2$ for all $y$. This follows for each $y$ from the assumption that $D-\mu$ is invertible. As $D-\mu$ is a linear operator with continuous inverse and by elliptic regularity~$H'$ is smooth in $x$ and $y$.

We set $G(x,y)=H(x,y)-H'(x,y)$, and this gives a Green function for $D-\mu$.

Assume that $G$ and $\tilde{G}$ are two Green functions for $D$, then
$(D_x-\mu)((G-\tilde{G})(.,y))=0$. By the invertibility, $G=\tilde{G}$ follows.
Smoothness follows by smoothness of all $G^{x_i}$, and smoothness of $F$ and
$H'$. 
\end{proof}

Note that due to the last Proposition Lemmata~\ref{adj.green}
and~\ref{compact_sym_Green} also hold true for manifolds $M$ of bounded geometry. \\

We finish this section by stating another property of the Green function:

\begin{lem}
 Let $(M,g)$ be a Riemannian spin manifold of bounded geometry, and let $D-\mu$ be invertible. Then the Green
function also decays in $L^2$ in the second entry, i.e., $G_{D-\mu}(x,.)\in
L^2(M\setminus B_r(x))$ for all $r>0$.
\end{lem}

\begin{proof} 
The Green function $G_{D-\bar{\mu}}(.,x)$ is in $L^2(M\setminus B_r(x))$ in the first component. Then the claim follows from Lemma~\ref{adj.green} in the extended version to manifolds $M$ of bounded geometry.
% In the proof of Proposition~\ref{prop.green.exists} the Green
% function is constructed as difference $H-H'$.  For $H$ we have $H(x,.)\in
% L^2(M\setminus B_r(x))$ since this is true for all Green function $G^{x_i}$ on
% the auxiliary closed manifolds $M_{x_i}$. Moreover,  
%  $H'$ fulfils $H'(x,y)=(D_x-\mu)^{-1} F(x,y)$ where $F(x,y)$ is a smooth section
% of $\Si_M \boxtimes \Si_M^*$ with compact support in $(M\times M)_{6R}$. Since
% $(D_x-\mu)^{-1}$ is bounded, $H'(x,.)\in L^2(M\setminus B_r(x))$ for all $r>0$.
% Thus, this also follows for the Green function.
\end{proof}

%%%%%%%%%%%%%%%%%%%%%%%%%%%%%%%%%%%%%%%%%%%%%%%%%%%
\section{The Dirac operator  
%Decay of the Green function 
on hyperbolic space and its products}\label{sec.D.hyp}
%%%%%%%%%%%%%%%%%%%%%%%%%%%%%%%%%%%%%%%%%%%%%%%%%%

In this section we examine the Dirac operator on the model spaces
$\Mc^{m,k}=\mH_c^{k+1}\times N$. Note that we also allow the case where $N$ is zero dimensional. First,
we introduce polar coordinates on $\mH_c^{k+1}$ and write the Dirac operator in these
coordinates. Then, we study the canonical action of $\Spin(k+1)$ on $\Mc^{m,k}$ and its
spinor bundle.

%%%%%%%%%%%%%%%%%%%%%%%%%%%%%%%%%%%%%%%%%%%%%%%%%%%%%%%%%%
\subsection{The Dirac operator in polar coordinates}
%%%%%%%%%%%%%%%%%%%%%%%%%%%%%%%%%%%%%%%%%%%%%%%%%%%%%%%%%%

Let us introduce some notation, and let us compare it to notation in the existing literature.

In this section we have to work with spinors on various submanifolds of $\mH_c^{k+1}\times N$.

So let $(Z_b)_{b\in B}$ a smooth family of pairwise disjoint submanifolds of $\mH_c^{k+1}\times N$. For simplicity of presentation we assume that all $Z_b$ are isomorphic to $Z$, in particular we obtain a smooth map $F\colon  Z\times B\to M$. The tangent space $TZ_b$ carries an induced connection and similar the normal bundle $\nu_b\to Z_b$ of $Z_b$ in $M$. As vector bundles $TM|_{Z_b}$ equals $\nu_b \oplus TZ_b$. The connection on those vector bundles are denoted by $\nabla^M$ for $TM|_{Z_b}$ and $\nabla^{\rm int}$ for $\nu_b \oplus TZ_b$. The difference is essentially the second fundamental form $\II_{Z_b}$ of $Z_b$ in $M$.

Putting all these bundles together for various $b$ we obtain the following bundles over $Z\times B$:
\[ F^*TM=\bigcup_{b\in B} TM|_{Z_b},\ TZ_B:=\bigcup_{b\in B} TZ_b,\ \nu_B:=\bigcup_{b\in B} \nu_b.\]
Again as bundles with scalar products we have $F^*TM=TZ_B\oplus \nu_B$ but both sides carry different metric connections. The pullback of Levi-Civita connection on $TM$ to $F^*TM$ is denoted by $\nabla^M$ whereas the sum connection on the right hand side is denoted by $\nabla^{\rm int}$ where for $X\in T_xZ_b$, $Y\in C^\infty (TZ_B)$ and $W\in C^\infty(\nu_B)$ we have 

\[\nabla_X^M Y-\nabla_X^{\rm int}Y= \II_{Z_b}(X,Y),\ \<\underbrace{ \nabla_X^M W- \nabla_X^{\rm int}W}_{\in T_xZ_b}, Y\>=-\<\II_{Z_b}(X,Y), W\>.\]

These two connection define connection 1-forms on the pullbacks of the frame bundle of $M$ and the spin structure of $M$. So we finally obtain connections, again denoted by $\nabla^M$ and $\nabla^{\rm int}$, on $F^*\Si_M\to Z\times B$.

In particular we have for all $X\in TZ_B$ and all spinors $\phi\in C^\infty(F^*\Si_M)$
\begin{equation}\label{spinor-induced}
\na^M_X \phi =  \na^\intr_X \phi + \frac12 \sum_i e_i\cdot
\II_Z(X,e_i)\cdot\phi 
\end{equation}
where $(e_i)_i$ is a local orthonormal frame on $F^*\Si_M$, cp. \cite[around (9)]{baer_98}.

\begin{rem} 
 In \cite{baer_98} a slightly different notation is used, as can be seen in the following dictionary of notations

 \begin{center}
\begin{tabular}{l||c|c|c|c|c|c}
Bär \cite{baer_98} & $Q$  & $M$ & $\na^Q$ and  & $\na^M\oplus \na^N$ and        
               & $\widehat D$ &  $\tilde D$\\
                   &      &     & $\na^{\Si Q}$ & $\na^{\Si M}\otimes \id + \id
\otimes \na^{\Si N}$ &  & \\
\hline
&&&&&&\\[-3mm]
Our article & $M=\mH_c^{k+1}\times N$  & $Z$ & $\na^M$ & $\approx \na^\intr$ &
$D^Z_\pa$ & $D_\intr^Z$
\end{tabular}
\end{center}\hfill\\

Furthermore, in \cite{baer_98} only the case that $B$ is a point is formally studied but the calculations in there immediately generalize to our setting.

Also be aware that in \cite{BGM} a further notation is used which has several advantages if the submanifold is a hypersurface which is not the case in our article. In \cite{BGM} the Clifford multiplication of the ambient manifold coincides with the Clifford multiplication on the submanifold only up to Clifford multiplication with the normal vector field. In contrast to this in our notation the Clifford multiplication of the ambient space $M$ coincides with the one on the submanifold $Z$.
\end{rem}

The partial Dirac operators
$D^Z_\pa$ are now defined as $D^Z_\pa=\sum_{i} e_i\cdot \nabla^M_{e_i}$, and  the
intrinsic Dirac operators are given by $D_\intr^Z=\sum_{i} e_i\cdot
\nabla^\intr_{e_i}$. As this definition does not depend on the choice 
of frame, it yields a global definition. The intrinsic Dirac operator is  
a twisted Dirac operator on the submanifold $N$. In the following applications all normal bundles have 
a parallel trivialization, hence, in this case the intrinsic Dirac operator
coincides on the submanifold with several copies of the Dirac operator on this
submanifold. As multiplicities are irrelevant for our discussion we have chosen the name 'intrinsic Dirac operator' for $D_{\rm int}$, slightly abusing the language.

By \eqref{spinor-induced}, the intrinsic Dirac operator $D_\intr^Z$ is related
to 
the partial Dirac operator $D_\pa^Z$ via 
  $$D_\pa^Z\phi =D_\intr^Z \phi - \frac12 \Hvec_Z\cdot \phi,$$ 
where $\Hvec_Z=\trace\, \II_Z$ is the unnormalised mean curvature vector field
of $Z$ in $M$, 
see \cite[Lemma~2.1]{baer_98}.\\

We now come to our specific situation $M=\Mc^{m,k}$: We express the hyperbolic
metric in polar normal coordinates centered in a fixed point $p_0$ which will be
sometimes identified  with $0$. In these polar coordinates $\Mc^{m,k}\setminus (\{p_0\}\times N)$ is parametrized by
$\mR^+\times \mS^k\times N$.

We are especially interested in the submanifolds $Z$ of $M=\Mc^{m,k}$ that are
either $\mR^+\times \{x\}\times \{y\}$ or $\{r\}\times \mS^k\times \{y\}$ or
$\{r\}\times \{x\}\times N^n$, always equipped with the metric induced from $\Mc^{m,k}$.
In the following we will address these families of submanifolds shortly by $\mR^+$, $\mS^k$ and
$N$. The corresponding spaces $B$ are then  $ \mS^k\times N$, $\mR^+\times N$ and $\mR^+\times \mS^k$, respectively. On an open set we choose an orthonormal frame $e_1,\ldots,e_m$, 
$m=n+k+1=\dim M$, such that $e_{k+2},\ldots,e_m$ is an orthonormal frame for 
the submanifolds $N$, and $e_{2},\ldots, e_{k+1}$ is an orthonormal
frame for  $\mS^k$ and where $e_1:=\pa_r$. The notation should be read such that
$\frac{\nabla}{dr}$ and $\pa_r$ denote 
essentially the same (radial) vector, but $\pa_r$ is viewed as a vector 
which acts via 
Clifford multiplication whereas $\frac{\nabla}{dr}$ acts as a  
covariant derivative.

The Dirac operator $D$ on $(r_0,\infty)\times \mS^k\times N$ 
is the sum of partial Dirac operators 
  $$D=  \pa_r \cdot \frac{\nabla}{dr}+ D_\pa^{\mS^k} + D_\pa^N $$

where the partial Dirac operators along $N$ and $\mS^k$
are locally defined as 
\begin{align*}
  D_\pa^N \phi := \sum_{i=1}^ne_i\cdot \nabla^M_{e_i} \phi, \quad 
  D_\pa^{\mS^k}\phi :=\sum_{i=n+1}^{n+k}e_i\cdot \nabla^M_{e_i} \phi,
\end{align*}
for $\phi\in C^\infty(\Si_M)$. 

The intrinsic Dirac operators along $N$ and $\mS^k$ are given by
\begin{align*}
  D_\intr^N \phi := \sum_{i=1}^ne_i\cdot \nabla^{\intr}_{e_i} \phi, \quad
  D_\intr^{\mS^k}\phi :=\sum_{i=n+1}^{n+k}e_i\cdot \nabla^{\intr}_{e_i} \phi.
\end{align*}

We denote the second fundamental form of $\mS^k$ in $\mH^{k+1}_c$ as
$\II_{\mS^k}$ and set $\Hvec_{\mS^k}:= \trace\, \II_{\mS^k}$. Then $\II_{\mS^k}$
and $\Hvec_{\mS^k}$ 
do not depend on whether they represent the second fundamental form and the mean
curvature field of $\mS^k$ in $\mH^{k+1}_c$, or of  
$\mS^k$ in $\mH^{k+1}_c\times N$ or of $\mS^k\times N$ in $\mH^{k+1}_c\times N$.

 Using $\Hvec_N=0$ and $f(r)=\sinh_c(r)$, cp. Section~\ref{sec_H}, 
\begin{equation*}
\Hvec_{\mS^k\times N}=\Hvec_{\mS^k}= -k \frac{\partial_r
f(r)}{f(r)}\pa_r=-k\coth_c(r)
\end{equation*}
we obtain $D^N:=D_\pa^N=D_\intr^N$ and
$D_\pa^{\mS^k}=D_\intr^{\mS^k}+\frac{k}2\coth_c(r)\pa_r\cdot$. 

We set $D^{\mS^k}:= f(r)D_\intr^{\mS^k}$ which is on each spherical submanifold up to 
multiplicity the standard Dirac operator on $\mS^k$ and
obtain
\begin{align}\label{D_dec}D= \frac{1}{\sinh_c(r)}D^{\mS^k}  + \pa_r\cdot  \frac{\nabla}{dr} +
\frac{k}2\coth_c(r)\pa_r\cdot + D^N.\end{align}

\begin{lemma}\label{commut.rel}
The following operators anticommute:
$D^N$ with $D^{\mS^k}$, $D^N$ with 
$\pa_r\cdot$, $D^N$ with $ \pa_r\cdot  \frac{\nabla}{dr}$, $D^{\mS^k}$ with  $ \pa_r\cdot$, and  $D^{\mS^k}$ with $\pa_r\cdot  \frac{\nabla}{dr}$. However $\pa_r\cdot$ 
commutes with  $ \pa_r\cdot  \frac{\nabla}{dr}$, and
 $(D^{\mS^k})^2$ commutes with $D$.
\end{lemma}

\begin{proof}
Let  $P_\Spin(\mH^{k+1}_c)\to P_\SO(\mH^{k+1}_c)$ and $P_\Spin(N)\to P_\SO(N)$
be the fixed spin structures on 
$\mH^{k+1}_c$ and $N$. Then we write as in Subsection \ref{spin_prod}
  \begin{align}\label{spin_prod_2} \Si_{\mH^{k+1}_c\times N}=
(\underbrace{P_\Spin(\mH^{k+1}_c)\times P_\Spin(N)}_P) \times_\zeta
\Si_m\end{align}
where $\zeta$ is the composition  
$\Spin(k+1)\times \Spin(n)\stackrel{\xi}{\to} \Spin(m)\stackrel{\rho_m}{\to}
\End(\Si_m)$. The bundle $P$ carries the Levi-Civita  connection-$1$-form
$\alpha^{LC}_{\Mc}$ and another  connection-$1$-form $\alpha^\intr$ as explained before. 

We obtain a connection preserving bundle homomorphism $I_c$, which is 
fiberwise an isometric isomorphism, and 
\begin{equation}\label{polar.spinor}
\begin{CD}
\Si_{\mH^{k+1}_c\setminus\{p_0\}\times N},\na^\intr   @>I_c>> \Si_{\mR^+\times
\mS^k\times N}, \na^{LC}\\ 
@VVV                                    @VVV\\ 
\mH^{k+1}_c\setminus\{p_0\}\times N       @>\id>> \mR^+\times \mS^k\times N    
\end{CD}
\end{equation}
commutes. Note that $I_c$ is also compatible with the Clifford multiplication
 in the sense that for $X\in TZ$ we have 
\[ I_c(X\cdot \phi)= \left\{\begin{matrix}
                             X\cdot I_c(\phi) & \text{for\ }Z=\mR^+\times
\{x\}\times \{y\} {\rm\ or\ } \{r\}\times\{x\}\times N\\
			     \frac{f(r)}{r} X\cdot I_c(\phi) & \text{for\ }Z=
\{r\}\times \mS^k\times \{y\}.
                            \end{matrix}
 \right. \] 

 Then the lemma follows immediately by the corresponding statements for
$\Si_{\mR^+\times \mS^k\times N}$. 
\end{proof}

We will also use the map $\hat{I}_c:=I_0^{-1}\circ I_c: \Si_{\mH^{k+1}_c\setminus\{p_0\}\times N}\to \Si_{\mR^{k+1}\setminus\{0\}\times N}$ which allows to identify $\Si_{\mH^{k+1}_c\times N}|_{(x,y)}$ with $\Si_{\mR^{k+1}\times N}|_{(x,y)}$ and thus with $\Si_{\mR^{k+1}\times N}|_{(0,y)}$, $0\cong p_0$.

%%%%%%%%%%%%%%%%%%%%%%%%%%%%%%%%%%%%%%%%%%%%%%%%%%%%%%%%%%%%%%%%%%%%%%%%%%%%
\subsection{The action of $\Spin(k+1)$ on $\Mc^{m,k}=\mH_c^{k+1}\times
N$}\label{spin_action}
%%%%%%%%%%%%%%%%%%%%%%%%%%%%%%%%%%%%%%%%%%%%%%%%%%%%%%%%%%%%%%%%%%%%%%%%%%%%
We identify $T_{p_0}\mH^{k+1}_c$ with~$\mR^{k+1}$.
The left action $a_1$ of the spin group $\Spin(k+1)$  on $\mR^{k+1}$
obtained by composing the double covering $\Spin(k+1)\to \SO(k+1)$ 
with the tautological representation yields a left action on $\mH^{k+1}_c$ via
the exponential map $\exp_{p_0}\colon  \mR^{k+1}\to \mH^{k+1}_c$ 
which is a diffeomorphism. 
As this action is isometric it yields a
left action on $P_\Spin(\mH^{k+1}_c)$ -- also called $a_1$. Thus, we obtain a $\Spin(k+1)$-action on
$P_\Spin(\mH^{k+1}_c)\times P_\Spin(N)\times \Sigma_m$ as $\hat{a}_1=a_1\times
\id\times \id$. Since $a_1$ and the principal $\Spin(k+1)$-action which acts from the right commute, the $\hat{a}_1$-action descends
to a $\Spin(k+1)$-action from the left -- denoted by $a_2$ -- on the spinor bundle 
$\Si_{\mH^{k+1}_c\times N}= (P_\Spin(\mH^{k+1}_c)\times P_\Spin(N)) \times_\zeta
\Si_m$ (for $\zeta$ as in \eqref{spin_prod_2}) such that 
 \[
\begin{CD}
\Si_{\mH^{k+1}_c\times N}  @>a_2(\gamma)>> \Si_{\mH^{k+1}_c\times N}\\ 
@VVV                                    @VVV\\ 
\mH^{k+1}_c\times N              @>a_1(\gamma)\times \id>> \mH^{k+1}_c\times N  
\end{CD}
\]
commutes. 

By construction, the action $a_1$ does not depend on $c$. Thus,
Diagram~\eqref{polar.spinor} commutes with this $\Spin(k+1)$-action. 

Moreover, note that $a_1$  preserves the spheres
$\mS_{r,y}^k:=\{r\}\times \mS^k \times \{y\}\subset \mH_c^{k+1}\times N$. Hence, the diagram above can be restricted to this submanifold. In particular,
$a_1$ acts transitively on $\mS_{r,y}^k$. Furthermore, $(p_0,y)$ is a fixed
point of $a_1\times \id$ for all $y\in N$. Thus, the $a_2$-action can be
restricted to an action that maps $\Si_{\mH_c^{k+1}\times N}|_{(p_0,y)}$ to
itself.

%%%%%%%%%%%%%%%%%%%%%%%%%%%%%%%%%%%%%%%%%%%%%%%%%%%%%%%%%%%%%%%%%%%%%%%%%%
\subsection{Spinors on $\mS^k\subset \mR^{k+1}$}
%%%%%%%%%%%%%%%%%%%%%%%%%%%%%%%%%%%%%%%%%%%%%%%%%%%%%%%%%%%%%%%%%%%%%%%%%%

We will now analyse the special case $N=\{y\}$ and $c=0$, thus
$\mH^{k+1}_c=\mR^{k+1}$.
This well-known case is not only important as an example, but will also be used
to derive 
consequences for the general case. 

We obtain immediately from \eqref{spinor-induced} and $\II_{\mS_r^k}= -\frac1r
g|_{\mS_r^k} \partial_r$ where $\mS_r^k$ is the sphere of radius $r$ canonically
embedded in $\mR^{k+1}$: 
\begin{lemma}\label{c=0_spin}
Assume that $\phi$ is a parallel spinor on $\mR^{k+1}$. Then for any $X\in
T\mS_r^k$ we have
  \[\na_X^\intr \phi = -\frac1{2r} \pa_r \cdot X \cdot \phi \ {\rm and\ } \na_X^\intr (\pa_r\cdot \phi) =  \frac1{2r}\pa_r\cdot X  \cdot  (\pa_r \cdot \phi) .\] 
\end{lemma}
In particular, we have 
  \[D^{\mS^k} \phi = r D^{\mS^k}_\intr \phi = -\frac{k}2 \pa_r\cdot \phi \ {\rm and\ }  D^{\mS^k} (\pa_r \cdot \phi)= -\frac{k}2 \pa_r\cdot (\pa_r\cdot  \phi).\]
Using Lemma~\ref{commut.rel} and $\nabla^\intr_X \pa_r=0$ this implies
  \[(D^{\mS^k})^2 \phi = \frac{k^2}4 \phi  \ {\rm and\ }   (D^{\mS^k})^2 (\pa_r\cdot  \phi) = \frac{k^2}4 (\pa_r\cdot\phi). \]

%%%%%%%%%%%%%%%%%%%%%%%%%%%%%%%%%%%%%%%%%%%%%%%%%%%%%%%%%%%%%%%%%%%%  
\section{Modes of $\Spin(k+1)$-equivariant maps} \label{sec_symm}
%%%%%%%%%%%%%%%%%%%%%%%%%%%%%%%%%%%%%%%%%%%%%%%%%%%%%%%%%%%%%%%%%%%%

We now have a $\Spin(k+1)$-action on $\Si_{\mR^{k+1}}|_{0}\cong \Si_{k+1}$, $\{r\}\times \mS^k$ and  $\Si_{\mR^{k+1}}|_{\{r\}\times \mS^k}$,  band thus one on $C^\infty(\mS^k, \Si_{\mR^{k+1}}|_{\{r\}\times \mS^k})$ given by $(\gamma\cdot f)(x)= a_2(\gamma)f(a_1(\gamma)^{-1} x)$. To simplify notations we mostly write $\mS^k$ for $\{r\}\times \mS^k$.

We now have to classify $\Spin(k+1)$-equivariant functions $\Si_{\mR^{k+1}}|_{0} \to C^\infty(\mS^k, \Si_{\mR^{k+1}}|_{\mS^k})$.

For $\psi_0\in \Si_{\mR^{k+1}}|_{0}$ let the parallel spinor on $\mR^{k+1}$ with value $\psi_0$ at $0$ be denoted by $\Psi_0$. For $k$ odd, the positive and negative parts of $\Psi_0$ are denoted by $\Psi_0^{(\pm)}$.

\begin{lemma}\label{ex_equ_1} 
Let $ F\colon  \Si_{k+1} \to C^\infty(\mS^k, \Si_{\mR^{k+1}}|_{\mS^k})$ be a $\Spin(k+1)$-equivariant map.  Then for $k$ even $F$ has the form
\[ \psi_0 \mapsto (a_1 \Psi_0 + a_2 \pa_r\cdot \Psi_0)|_{\mS^k}
\]
and for $k$ odd $F$ has the form 
\[ \psi_0 \mapsto (a_{11} \Psi_0^{(+)} + a_{22} \Psi_0^{(-)} +  a_{21} \pa_r\cdot \Psi_0^{(+)} + a_{12} \pa_r\cdot \Psi_0^{(-)})|_{\mS^k} 
\]
for suitable constants $a_i, a_{ij}\in\mC$.
\end{lemma}

\begin{proof}
First, we note that the maps $F$ described above are actually $\Spin(k+1)$-equivariant since $\pa_r$ is a $\Spin(k+1)$-equivariant vector field.

Let $A\colon \Si_{k+1}^{(\delta)}\hookrightarrow  \Si_{\mR^{k+1}}|_{0}$ be the inclusion map, $\delta\in\{+,-\}$. By composition we obtain for fixed  $\delta, \ep \in \{+, - \}$ a
$\Spin(k+1)$-equivariant map
\begin{align}
  \Si_{k+1}^{(\delta)}\stackrel{A}{\to}&\ \Si_{\mR^{k+1}}|_{0}
\stackrel{F}{\to}  C^\infty(\mS^k, \Si_{\mR^{k+1}}|_{\{r\}\times \mS^k})  \to C^\infty(\mS^k,\Si_{k+1}^{(\ep)}),\label{k+1_map_1}\end{align}
where in the last step we projected $\Si_{k+1}$ to $\Si_{k+1}^{(\ep)}$.
If we compose this map with evaluation at the north pole of the
sphere, then we obtain a $\Spin(k)$-equivariant map  $\sigma\colon  \Si_{k+1}^{(\delta)}\to \Si_{k+1}^{(\ep)}$. Because of the $\Spin(k+1)$-equivariance
of \eqref{k+1_map_1} and since $\Spin(k+1)$ acts transitively on $\mS^k$, this map
uniquely determines the map  
$\Si_{k+1}^{(\delta)}\to C^\infty(\mS^k,\Si_{k+1}^{(\ep)})$ above.

If $k$ is odd, then 
$\Si_{k+1}^{(\ep)}\cong \Si_{k+1}^{(\delta)}\cong \Si_k$ as 
$\Spin(k)$-representations, and Schur's Lemma tells us 
that there is up to scaling a unique such map $\sigma$.
Using the fact that $e_{k+1}\cdot\colon  \Si_{k+1}^{(\pm)}\to \Si_{k+1}^{(\mp)}$, $\sigma$ can be written as
\begin{align*}
\tau\in \Sigma_{k+1}^{(\delta)}\mapsto \left\{
\begin{matrix}
a_{\delta,\delta} \tau  & {\rm for\ }\delta=\ep\\
a_{\delta,\ep}  e_{k+1}\cdot \tau  & {\rm for\ }\delta\neq\ep
\end{matrix}
\right.\in \Sigma_{k+1}^{(\ep)}
\end{align*}
for suitable $a_{\delta, \ep}\in \mC$. As $\pa_r$ is the $\Spin(k+1)$-equivariant extension of $e_{k+1}$ we obtain the lemma for $k$ odd.

If $k$ is even, then 
$\Si_{k+1}=\Si_k= \Si_{k}^{(+)}\oplus \Si_{k}^{(-)}$ as 
$\Spin(k)$-representations. In this case $e_{k+1}\cdot$ commutes with $\Spin(k+1)$ and preserves the splitting. Using Schur's Lemma, $e_{k+1}^2=-1$ and because $e_{k+1}\cdot$ is tracefree we know that $e_{k+1}\cdot$ acts as $\pm {\rm diag}(\i, - \i)$. For $\ep=\delta$ we can
again apply Schur's Lemma. As $\Si_{k}^{(+)}$ and $\Si_{k}^{(-)}$ are not equivalent as $\Spin(k)$-representations the maps  
$\sigma\colon  \Si_{k}^{(\pm)}\to \Si_{k}^{(\mp)}$ have to be identically zero. Thus, with respect to the
splitting of $\Si_{k+1}$ the maps $\sigma$ for different $\delta$ and $\epsilon$ form a $\Spin(k)$-equivariant map $\Si_{k+1}\to \Si_{k+1}$
that can be written as
  $$\begin{pmatrix}b_{1} & 0\\ 0 & b_{2}\end{pmatrix}= \frac{b_1+b_2}{2}\Id +
\frac{b_1-b_2}{2}\begin{pmatrix}1 & 0\\ 0 & - 1\end{pmatrix}$$
for suitable $b_{i}\in\mC$.
Summarizing, for $k$ even, $\sigma$ maps $\tau\mapsto a_1\tau +a_2 e_{k+1}\cdot \tau$ with $a_i\in \mC$.
\end{proof}

Then using Lemma~\ref{c=0_spin} we obtain immediately

\begin{cor}\label{cor_ex}
 Let $ F\colon  \Si_{k+1} \to C^\infty(\mS^k, \Si_{\mR^{k+1}}|_{\mS^k})$ be a $\Spin(k+1)$-equivariant map.  Let $\psi_0\in \Si_{k+1}$ and $\phi=F\psi_0$. Then $(D^{\mS^k})^2\phi=\frac{k^2}{4}\phi$.
\end{cor}

We say that $\phi$ is in the spherical mode $\frac{k^2}{4}$, and thus $\phi$ is in the mode of lowest energy on the sphere.\\

Now we want to carry over the last result to $\Mc^{m,k}$. In the following $p_0\in \mH_c^{k+1}$ denotes again the fixed point of the $\Spin(k+1)$-action, and let $y_0, y\in N$.

\begin{lemma}\label{ex_equ_2}
Let $ F\colon  \Si_{\mH_c^{k+1}\times N}|_{(p_0, y_0)} \to C^\infty(\mS^k, \Si_{\mH_c^{k+1}\times N}|_{\mS^k\times \{y\}})$ be a $\Spin(k+1)$-equivariant map.  Let $\psi_0\in \Si_{\mH_c^{k+1}\times N}|_{(p_0, y_0)}$ and $\phi=F\psi_0$. Then $(D^{\mS^k})^2\phi=\frac{k^2}{4}\phi$.
\end{lemma}

\begin{proof}
Note that the composition $\hat{I}_c:=I_0^{-1}\circ I_c$ where $I_c$ is defined as in 
\eqref{polar.spinor} maps the spinor bundle over 
$(\mH_c^{k+1}\setminus\{p_0\})\times N$ to the spinor bundle
over $(\mR^{k+1}\setminus\{0\})\times N$. This map preserves the 
intrinsic connection $\na^\intr$ and uniquely extends into $p_0\cong 0$.
Via pullback we then obtain a $\Spin(k+1)$-equivariant vector space isomorphism
 \begin{equation*}%\label{sok1inv}
C^\infty( \{r\}\times \mS^k,  \Si_{\mH_c^{k+1}\times N}|_{\{r\}\times
\mS^k\times \{y\}}) 
  \stackrel{J_{r,y}}{\to} C^\infty( \{r\}\times \mS^k, \Si_{\mR^{k+1}\times N}|_{\{r\}\times
\mS^k\times \{y\}}),\ 
  \psi\mapsto \hat{I}_c \circ\psi.
\end{equation*}
Moreover, we can write in the sense of $\Spin(k+1)$-modules $\Si_{\mR^{k+1}\times N}|_{(x, y)}\cong\Si_m\cong \Si_{k+1}\otimes V$ if $k$ is even
or $\Si_{\mR^{k+1}\times N}|_{(x, y)}\cong \Si_{k+1}^{(+)}\otimes V^{(+)}
\oplus \Si_{k+1}^{(-)}\otimes V^{(-)}$ if $k$ is odd, where 
$V^{(\ep)}:=\Hom_{\Spin (k+1)}(\Si_{k+1}^{(\ep)}, \Si_{\mR^{k+1}\times N}|_{(x, y)})$ is a vector space
which is independent of $x\in \mR^{k+1}$. 

Let now $k$ be odd. Then any $\al\in (V^{(\ep)})^*$ defines a map $\Si_{\mH_c^{k+1}\times N}|_{(x, y)}\to
\Si_{k+1}^{(\ep)}$.

  Let $A\colon \Si_{k+1}^{(\delta)}\to  \Si_{\mH^{k+1}_c\times N}|_{(p_0,y_0)}$
be a $\Spin(k+1)$-equivariant map. By composition we obtain for fixed  $A$, $\alpha$ and $\delta, \ep \in \{+, - \}$ a
$\Spin(k+1)$-equivariant map
\begin{align*}
  \Si_{k+1}^{(\delta)}\stackrel{A}{\to}&\ \Si_{\mH^{k+1}_c\times N}|_{(p_0,y_0)}
\stackrel{F}{\to}  C^\infty(\mS^k, \Si_{\mH_c^{k+1}\times N}|_{\mS^k\times \{y\}})\nonumber\\
   \stackrel{J_{r,y}}{\to}&   C^\infty(\mS^k, \Si_{\mR_c^{k+1}\times N}|_{\mS^k\times \{y\}}) \cong C^\infty(\mS^k, \Si_{\mR^{k+1}}|_{\mS^k}\otimes V)   \stackrel{\alpha}{\to} 
  C^\infty(\mS^k,\Si_{k+1}^{(\ep)}).
  \end{align*}
Let now $k$ be even. Then the argumentation is analogous to the one above when replacing $V^{(\ep)}$ by $V$ and  $\Si^{(\ep)}_{k+1}$ by  $\Si_{k+1}$.

Then the Lemma follows from Corollary~\ref{cor_ex} together with the identification by $J_{r,y}$.
\end{proof}

\begin{cor}\label{sym_Green}
 Let $G(q,p)$ be the Green function of the operator $D-\mu$, $\mu\not\in
\Spec_{L^2}^{\Mc}(D)$. Let $q=(r,x,y)\in \Mc^{m,k}$ be the polar coordinates when using
$p_0$ as the origin, $r>0$. Let $\psi_0\in \Si_{\Mc^{m,k}}|_{(p_0, y_0)}$, $y_0\in N$.
Set $\phi(q):=G(q,(p_0, y_0))\psi_0$. Then  
  $$ (D^{\mS^k})^2 \phi|_{\{r\}\times \mS^k\times \{y\}}=  \frac{k^2}4
\phi|_{\{r\}\times \mS^k\times \{y\}}.$$
\end{cor}

\begin{proof} 
Now we consider the Green function of the shifted Dirac operator $D-\mu$ for
$\mu\not\in \spec_{L^2}^{\Mc}(D)$ as a map 
  \[G(\,\cdot\,,(p_0,y_0))\colon \Sigma_{ \mH^{k+1}_c\times N}|_{(p_0,y_0)} \to 
   \Gamma( \Sigma_{\mH^{k+1}_c\times N\setminus\{(p_0,y_0)\}}).\]
It follows from the definition of $G$, in particular from its uniqueness, and
from the equivariance of $D$ under $\Spin(k+1)$ that $G(\,\cdot\,,(p_0,y_0))$ is
$\Spin(k+1)$-equivariant as well.
 In particular, $G(\,\cdot\,,(p_0,y_0))|_{\mS^k\times \{y\}}$ is a $\Spin(k+1)$-equivariant map as considered in
 Lemma~\ref{ex_equ_2}. Thus, together with Lemma~\ref{ex_equ_2}  the corollary follows.
\end{proof}

\section{Decay estimates for a fixed mode}\label{sec.decay}
%%%%%%%%%%%%%%%%%%%%%%%%%%%%%%%%%%%%%%%%%%%%%%%%%%%

Let $\mu\not\in\Spec_{L^2}^{\Mc^{m.k}}(D)$. Then, by
Theorem~\ref{prop.green.exists} there exists a unique Green function for
$D-\mu$. The goal of this section is to estimate the decay of this Green
function at infinity. For that, let $y=(p_0,y_N)\in \mH_c^{k+1}\times N$ and $\psi_0\in \Sigma_{\Mc}|_{y}$
be fixed. Set $\phi(x):=G(x,y)\psi_0$. The Definition of the Green
function, cf. \eqref{def.green}, implies that $\phi$ is an $L^2$-eigenspinor of $D$ to the eigenvalue $\mu$ outside a
neighbourhood of~$y$. Moreover,  by Corollary~\ref{sym_Green} we know that $\phi$
is in the spherical mode $\frac{k^2}{4}$.

Before starting to estimate the decay we give the following Remark:

\begin{rem}\label{specL2_MC}
The $L^2$-spectrum of the square of the Dirac operator on the product space
$M_1\times M_2$ is given by the set theoretic sum $\{\lambda_1^2 +\lambda_2^2\
|\ \lambda_i^2\in \Spec_{L^2}^{M_i}((D^{M_i})^2)\}$. This is seen immediately by
\eqref{Dirac_square_prod} and the spectral theorem.

 The $L^2$-spectrum of the Dirac operator on the hyperbolic space, and thus also
on $\mH_c^{k+1}$, is the whole real line, cf. \cite{bunke_91}. Let
$\lambda_0^2$, $\lambda_0\geq 0$, be the smallest eigenvalue of $(D^N)^2$.  Then
the above together with Lemma~\ref{spec_D2}  implies for $\Mc^{m,k}$ that \[
\Spec_{L^2}^{\Mc}(D^2)=[\lambda_0^2, \infty). \]

Together with Lemma~\ref{spec_or} and Example~\ref{spec_or_ex} \[
\Spec_{L^2}^{\Mc}(D)=(-\infty, -\lambda_0]\cup [\lambda_0, \infty).\] The
complement of this spectrum is denoted by $I_{\lambda_0}:=(\mC\setminus \mR)\cup
(-\lambda_0,\lambda_0)$.
\end{rem}

Now we decompose the space of spinors restricted to $\{r_1\}\times \mS^k\times
N$
into complex subspaces of minimal dimensions which are invariant under 
$D^N$, $\pa_r\cdot$, $D^{\mS^k}$. Such spaces have a basis of the form $\psi$, 
$\pa_r\cdot\psi$, $P\psi$, and $\pa_r\cdot P\psi$, where
$\psi$ satisfies $D^N\psi=\lambda \psi$, $(D^{\mS^k})^2\psi =\rho^2\psi$, 
$\rho \in \frac{k}2+ \mN_0$, and $P:=D^{\mS^k}/\rho$.
All these operations commute with parallel transport in $r$-direction,
so by applying parallel transport in $r$-direction
we obtain spinors
$\psi$, $\pa_r\cdot\psi$, $P\psi$, and 
$\pa_r\cdot P\psi$ on $\mR^+\times \mS^k\times N$ with similar relations,
and the space of all spinors of the form
\begin{equation}\label{spinor.form}
  \phi= \phi_1(r)\psi + \phi_2(r)\pa_r\cdot \psi +  \phi_3(r) P\psi + \phi_4(r)
\pa_r\cdot P\psi
\end{equation}
is preserved under the Dirac operator $D$ on $\Mc^{m,k}$ because of \eqref{D_dec}. 
Then the operators discussed above restricted to such a minimal subspace are 
represented by the matrices, cp. Lemma~\ref{commut.rel},
  \begin{equation*}
    D^N =
    \begin{pmatrix}
      \la & 0 & 0 & 0\\
       0 & -\la & 0 & 0\\
       0 & 0 & - \la & 0\\
       0 & 0 & 0 & \la
    \end{pmatrix}\qquad
  %\end{equation*}
  %\begin{equation*}
    D^{\mS^k} =
    \begin{pmatrix}
       0 & 0 & \rho & 0\\
       0 & 0 & 0 & -\rho\\
       \rho & 0 & 0 & 0\\
       0 & -\rho & 0 & 0
    \end{pmatrix}\qquad
  %\end{equation*}
  %\begin{equation*}
    \pa_r\cdot =
    \begin{pmatrix}
       0 &-1 & 0 & 0\\
       1 & 0 & 0 & 0 \\
       0 & 0 & 0 &-1\\
       0 & 0 & 1 & 0
    \end{pmatrix}
  \end{equation*}

\begin{prop}\label{decay_mode}
Assume that $\phi$ is an $L^2$-solution
to the equation $D\phi=\mu\phi$, $\mu\in I_{\lambda_0}$ on 
$(\Mc^{m,k})_{>r_0}:=(\mH_c^{k+1}\setminus B_{r_0}(p_0))\times N$. Assume that $\phi$
has the form given in \eqref{spinor.form} with parameters $\rho$ and $\lambda$.
Let $\kappa$ satisfy $\kappa^2=\la^2-\mu^2$, 
$\Re \kappa\geq 0$. Then  $\Re \kappa>0$.
% , and for all $r_1$ sufficiently large
% there is a constant $C_\la$ such that for $c\neq 0$
% \begin{equation*}
% |\phi(x)|\leq C_\lambda \|\phi\|_{L^2((\Mc^{m,k})_{>r_0})}e^{(-ck/2-\Re \kappa)
% d(x_1,p_0)}\qquad\mbox{for all }
%     x=(x_1,x_2)\in (\mH_c^{k+1}\setminus B_{r_1}(p_0))\times N
% \end{equation*}
% where $C_\la$\mnote{n\"otig?} is a constant that only depends on $c, k,r_1,\la, \mu$ and $\rho$.
Moreover, let  $\kappa_{\la_0}^2=\la_0^2-\mu^2$. If  $\Re \kappa_{\la_0}> 0$, then
there is are positive constants $C$ and $r_1$ such that
\begin{equation*}
|\phi(x)|\leq C \|\phi\|_{L^2((\Mc^{m,k})_{>r_0})}e^{(-ck/2-\Re \kappa_{\la_0})
d(x_1,p_0)}\ \ \mbox{for all }
    x=(x_1, x_2)\in (\mH_c^{k+1}\setminus B_{r_1}(p_0))\times N
 \end{equation*}
where $C$ is a constant that only depends on $c, k,r_1,\la_0, \mu$ and $\rho$
but not on $\lambda$. 
For $c=0$ an analogous estimate holds when replacing $e^{-(ck/2)d(x_1,p_0)}$ by $r^{-k/2}$ where $r=d(x_1,p_0)$.
\end{prop}

\begin{proof}
By assumption $\phi$ can be written as in \eqref{spinor.form}.  We view the
components of $\phi$ as a vector in $\mC^4$, 
i.e., $\Phi(r):=(\phi_1(r),\phi_2(r),\phi_3(r),\phi_4(r))$. So by \eqref{D_dec} the following
equation is equivalent to $D\phi=\mu\phi$:
  \begin{equation*}
     0 =
    \begin{pmatrix}
       \la-\mu & - \frac{k}2 \coth_c r & \frac\rho{\sinh_c r} & 0\\
       \frac{k}2 \coth_c r  & -\la-\mu & 0 & -\frac\rho{\sinh_c r} \\
       \frac\rho{\sinh_c r} & 0 & -\la-\mu &  -\frac{k}2 \coth_c r \\
       0 &- \frac\rho{\sinh_c r} & \frac{k}2 \coth_c r  & \la-\mu
    \end{pmatrix}
    \Phi(r) +
   \begin{pmatrix}
       0 &-1 & 0 & 0\\
       1 & 0 & 0 & 0 \\
       0 & 0 & 0 &-1\\
       0 & 0 & 1 & 0
    \end{pmatrix}
    \Phi'(r).
  \end{equation*}
Thus using $\1$ for the identity matrix and setting 
\begin{equation*}
   A:=     
   \begin{pmatrix}
       0  & \la+\mu & 0 & 0 \\
       \la-\mu & 0 & 0 & 0\\
       0 & 0 & 0 & -\la+\mu\\
       0 & 0 & -\la-\mu &  0 \\
   \end{pmatrix},
   \qquad 
  B:=   \begin{pmatrix}
       0  & 0 & 0 & 1 \\
       0  & 0 & 1 & 0 \\
       0  & 1 & 0 & 0 \\
       1  & 0 & 0 & 0 \\
   \end{pmatrix},
\end{equation*}
we obtain

\begin{equation*}
   \Phi'(r)=  \left(A  -\frac{k\coth_c r}2 \1 +  \frac\rho{\sinh_c r} B\right)
   \Phi(r).
\end{equation*}
%\begin{equation*}
%   \Phi'(r)=     
%   \begin{pmatrix}
%       -\frac{k}2 \coth r  & \la+\mu & 0 & \frac\rho{\sinh r} \\
%       \la-\mu & - \frac{k}2 \coth r & \frac\rho{\sinh r} & 0\\
%       0 & \frac\rho{\sinh r} & -\frac{k}2 \coth r  & -\la+\mu\\
%       \frac\rho{\sinh r} & 0 & -\la-\mu &  -\frac{k}2 \coth r \\
%   \end{pmatrix}
%   \Phi(r)
%\end{equation*}

We start with the case $c\neq 0$: 
We now substitute $t=e^{-cr}$, $\witi\Phi(t)=\Phi(-c^{-1}\log t)$. Then
\begin{align*}
  \frac{d\witi\Phi}{dt} &=  \left(-\frac{1}{ct} A + \frac{k(1+t^2)}{2(t-t^3)}\1
+ \frac{2\rho}{t^2-1} B\right) \witi\Phi.
\end{align*}
Such singular ordinary differential equations are well understood, 
see \cite[Chap. 4, Sec. 1--3]{coddington.levinson:55}. In particular, $t=0$
is a singular point of first kind, and \cite[Chap. 4 Thm
2,1]{coddington.levinson:55} yields that $t=0$ is a so-called ``regular singular
point'', 
and the associated theory applies. However, in our situation
it is more efficient to analyse the equation directly.

We set $h(t):= \bigl(\log t -\log (t+1) -\log (1-t)\bigr) k/2$, then 
$h'(t)=  \frac{k(1+t^2)}{2(t-t^3)}$. We define 
  $$\wihat\Phi(t):=e^{-h(t)} t^{A/c} \witi\Phi(t),$$
and we calculate
\begin{align*}
  \frac{d\wihat\Phi}{dt} &= - \frac{2\rho}{1-t^2} t^{A/c} B t^{-A/c}
\wihat\Phi.
\end{align*}
As $B$ anticommutes with $A$, we have $t^{A/c} B t^{-{A/c}}=
t^{2{A/c}}B$, 
and as $B$ is an isometry of $\mC^4$, we see that 
\begin{equation*}
\|t^{A/c} B t^{-{A/c}}\|= t^{2|\Re \kappa_+|/c} 
\end{equation*}
where $\|\,.\,\|$ denotes the operator norm and where 
  $$\kappa_\pm:=\pm \sqrt{\la^2-\mu^2}.$$
are the (complex) eigenvalues of $A$.
It follows that for $0<t<1/2$
\begin{equation*}
  \left|\frac{d}{dt}\log|\wihat\Phi(t)| \right| \leq
\frac{|\frac{d}{dt}\wihat\Phi|}{|\wihat\Phi|} \leq 
\frac{2\rho}{1-t^2}\|t^{A/c} B t^{-{A/c}} \| \leq 3 \rho
t^{2|\Re \kappa_+|/c}.
\end{equation*}
Thus the solution extends to $t=0$, and 
\begin{equation*}
   |\wihat\Phi(0)| e^{-3\rho t^{2|\Re \kappa_+|/c}}\leq
|\wihat\Phi(t)|\leq |\wihat\Phi(0)| e^{3 \rho t^{2|\Re \kappa_+|/c}}.
\end{equation*}
This estimate yields explicit asymptotic control for $\wihat\Phi(t)$, and thus
for $\phi$. Namely, assume $cr_0\geq 1> \log 2$, there are two fundamental
solutions
$\phi_\pm$ of $D\phi_\pm=\mu\phi_\pm$ such that $\wihat\Phi_\pm(0)$ is an
eigenvector of $A$ to the eigenvalue $\kappa_\pm$ and such that 
\begin{equation*}
 e^{-3 \rho e^{-2|\Re \kappa_+|r}} e^{\Re \kappa_\pm r} e^{h(e^{-cr})}   \leq
\frac{|\phi_\pm (x)|}{|\wihat\Phi_\pm(0)|}\leq  e^{3 \rho e^{-2|\Re 
\kappa_+|r}}e^{\Re \kappa_\pm r}e^{h(e^{-cr})}   \qquad r:=d(x_1,p_0)>r_0.
\end{equation*}
This implies that for every $\delta\in (0,1)$ there is $\tilde{r}_0$ such that
\begin{equation}\label{decay_1}
(1-\delta)  e^{(-(ck/2) +\Re \kappa_\pm) r}  \leq \frac{|\phi_\pm
(x)|}{|\wihat\Phi_\pm(0)|}\leq (1+\delta)  e^{(-(ck/2)+ \Re \kappa_\pm) r}   
\qquad r:=d(x_1,p_0)>\tilde{r}_0.
\end{equation}
From 
 $$ \int_{\tilde{r}_0}^\infty |\Phi(r)|^2 (\sinh_c r)^k\,\d r\leq
\frac{\|\phi\|_{L^2}^2}{\vol(\mS^k)\vol(N)}$$
and the left inequality of \eqref{decay_1}
we see that $\phi_\pm$ is in $L^2((\Mc^{m,k})_{>\tilde{r}_0})$ if and only of
$\Re \kappa_\pm<0$. In the following we call this $\kappa_\pm$ just
$\kappa_\lambda$  and also replace the $\pm$ index by $\lambda$ in all other
occurrences. We note that $|\Re \kappa_\la|$ is increasing in $|\lambda|$. Thus,
$\delta$ and $\tilde{r}_0$ from above can be chosen independent on $\lambda$.

Next, we multiply the first inequality of  \eqref{decay_1} by
$|\wihat\Phi_\lambda (0)|$ and then integrate its square:

\begin{align*}\frac{\|\phi\|_{L^2}^2}{\vol(\mS^k\times N)}\geq &
(1-\delta)^2|\wihat\Phi_\lambda (0)|^2  \int_{\tilde{r}_0}^\infty e^{(-ck +2\Re
\kappa_\lambda) r} (\sinh_c r)^k\,\d r.
\end{align*}

Hence, we obtain an upper bound  
\begin{align*}|\wihat\Phi_\lambda(0)|^2\leq C_1^2(1-\delta)^{-2}\Vert
\phi\Vert^2_{L^2((\Mc^{m,k})_{>\tilde{r}_0})} \left(\frac{e^{2\Re \kappa_\la
\tilde{r}_0}}{-2\Re \kappa_\la}\right)^{-1}
\end{align*}
where $C_1$ is a constant independent on $\lambda$.

Using this again with the right inequality of \eqref{decay_1} we get for all $x$
with $r=\dist(x, p_0)> \tilde{r}_0$ that 
\begin{align} |\phi(x)|\leq &\frac{1+\delta}{1-\delta}
C_1\|\phi\|_{L^2((\Mc^{m,k})_{>\tilde{r}_0})} e^{(-ck/2 +\Re  \kappa_{\lambda}) r} 
\left(\frac{e^{2\Re \kappa_\la \tilde{r}_0}}{-2\Re
\kappa_\la}\right)^{-1/2}\nonumber\\
 \leq & C_1 (-2\Re \kappa_\la)^\frac{1}{2} \|\phi\|_{L^2((\Mc^{m,k})_{>\tilde{r}_0})}
e^{-ckr/2 +\Re  \kappa_{\lambda} (r-\tilde{r}_0)}\label{equation_1}.
 % \leq & C_\lambda  \|\phi\|_{L^2((\Mc^{m,k})_{>\tilde{r}_0})} e^{(-ck/2 +\Re 
%\kappa_{\lambda}) r}\nonumber
\end{align}
%where $C_\lambda$ depends on $c, k, \mu, \rho, \tilde{r}_0$ and $\lambda$.
For $r>\tilde{r}_0$ we see that 
$(-2\Re \kappa_\la)e^{2\Re  \kappa_{\lambda} (r-\tilde{r}_0)}$ is monotonically
decreasing in $|\Re \kappa_\la|$, and we obtain from \eqref{equation_1}
\begin{align*} |\phi(x)|\leq & C_1 (-2\Re \kappa_{\la_0})^\frac{1}{2}
\|\phi\|_{L^2((\Mc^{m,k})_{>\tilde{r}_0})} e^{-ckr/2 +\Re  \kappa_{\lambda_0}
(r-\tilde{r}_0)} \\
\leq & C \|\phi\|_{L^2((\Mc^{m,k})_{>\tilde{r}_0})} e^{(-ck/2 +\Re 
\kappa_{\lambda_0}) r}
\end{align*}

for all $x$ with $r=d(x_1,p_0)>\tilde{r}_0$.
Here, $C$ can be chosen such that it only depends on $c$, $k$, $\tilde{r}_0$,
$\lambda_0$, $\mu$ and  $\rho$ but not on $\lambda$. 
Note that the $\kappa$ in the claim is simply $-\kappa_{\lambda_0}$.

It remains the case $c=0$:
\[ \Phi'(r)=\left( A-\frac{k}{2r}\1+\frac{\rho}{r}B \right)\Phi(r).\]
Set $\hat{\Phi}(r)=r^{\frac{k}{2}}e^{-Ar}\Phi(r)$. Then, $\hat{\Phi}'(r)= \frac{\rho}{r} e^{-Ar} B e^{Ar} \hat{\Phi}=\frac{\rho}{r} e^{-2Ar} B \hat{\Phi} $.
Then we can proceed as above and obtain the claim.
\end{proof}

In order to estimate the decay of  $\phi(x)=G(x,y)\psi_0$, $\psi_0\in \Si_{\Mc^{m,k}}|_{y}$ at infinity we will decompose $\phi$ into its modes in $\mS^k$ and $N$ direction, respectively. Lemma~\ref{decay_mode} provides an estimate of the decay of each mode which is independent of the mode in direction of $N$. Moreover, from Corollary~\ref{sym_Green} we know that $\phi$ has spherical mode $\frac{k^2}{4}$. Thus, we obtain a decay estimate for $\phi$:

\begin{lemma}\label{decay_Green}
Let $\mu\not\in \Spec_{L^2}^{\Mc}(D)$,  and let $G$ be the unique Green function
of $D-\mu$. We set $M_y(r):= \{x\in \Mc\ |\ \dist(x, N^y)
=r\}$ where $N^y = \{p_0\}\times N$  and $y=(p_0,y_N)\in \mH_c^{k+1}\times N$. Let $\kappa$ satisfy  $\kappa^2=\la_0^2-\mu^2$
and 
$\Re\kappa\geq 0$. Then for all $\ep>0$ and $r_0$ sufficiently large there is
a constant $C>0$ independent on $y$ such that
\[\int_{M_y(r)} |G(x,y)|^2 dx\leq Ce^{-2r\, \Re\!\kappa}\ \text{\ for\ all\ }
r>r_0.\]
\end{lemma}

\begin{proof}
 Let $\psi_0 \in \Si_{\Mc}|_{y}$. Set $\phi(x):=G(x,y)\psi_0$. Then, $\phi$
decomposes into a sum of spinors  $\phi_{\rho^2,\lambda, \nu}$   of the form
\eqref{spinor.form} with
$(D^{\mS^k})^2\phi_{\rho^2,\lambda, \nu}=\rho^2\phi_{\rho^2,\lambda, \nu}$ and $D^N\phi_{\rho^2,\lambda, \nu}=\lambda\phi_{\rho^2,\lambda, \nu}$, respectively, and as the multiplicities of the combined eigenspaces might be larger than one the index $\nu$ runs through a basis. By Corollary~\ref{sym_Green} $\rho^2$ may only take the value
$\frac{k^2}{4}$. Thus, $\int_{M_y(r)} |\phi(x)|^2 dx = \sum_{\lambda, \nu} \Vert \phi_{k^2/4,
\lambda, \nu}\Vert_{L^2(M_y(r))}^2$.

Together with Proposition \ref{decay_mode} we obtain for $c\neq 0$
\begin{align*}
 \int_{M_y(r)} |\phi(x)|^2 dx &\leq  \sum_{\lambda, \nu} C \Vert \phi_{k^2/4,
\lambda, \nu}\Vert_{L^2((\Mc^{m,k})_{>r_0})}^2 e^{(-ck-2 \Re \kappa)r} \sinh_c^k(r)\\
&\leq C' e^{-2r \Re \kappa } \sum_{\lambda, \nu}  \Vert \phi_{k^2/4, \lambda,
\nu} \Vert_{L^2((\Mc^{m,k})_{>r_0})}^2\\
&\leq C' e^{-2r\, \Re\!\kappa} \Vert \phi \Vert_{L^2((\Mc^{m,k})_{>r_0})}^2.
\end{align*}

The case $c=0$ follows analogously.
 \end{proof}

%%%%%%%%%%%%%%%%%%%%%%%%%%%%%%%%%%%%%%%%%%%%%%%%%%%%%%%%%%%%%%%%%%%%%
%
% Decomposition of G
%
%%%%%%%%%%%%%%%%%%%%%%%%%%%%%%%%%%%%%%%%%%%%%%%%%%%%%%%%%%%%%%%%%%%%%

\section{ Decomposition of the Green function}\label{sec_Green}

We decompose the Green function $G$ of the shifted Dirac operator $D-\mu$ on
$M=\Mc^{m,k}$ into a singular part and a smoothing operator.  Both operators
will be shown to be bounded operators from $L^p$ to $L^p$ for all $p\in
[1,\infty]$. 

At first we
choose a smooth cut-off function $\chi\colon \R\to [0,1]$ with $\supp\,
\chi\subset [-R,R]$ and $\chi_{|_{(-R/2,R/2)}}\equiv 1$. Let $\rho\colon  M\times M\to [0,1]$ be given by $\rho(x,y)=\chi
(\dist_{\mH_c^{k+1}}(\pi_{\mH}(x),\pi_{\mH}(y)))$.

Let now 
\[G_1(x,y):=\rho(x,y)G(x,y)\ {\rm and\ } G_2(x,y):=G(x,y)-G_1(x,y).\]
Then $G_2$ is zero on a neighbourhood of the diagonal, 
and thus smooth everywhere. The
singular part is only contained in $G_1$.

\begin{prop} \label{Green_sing}
Let $M=\Mc^{m,k}$ and $G_1$ be as defined above.  Then, for all $1\leq p\leq
\infty$ the map $P_1\colon \phi\mapsto \int_M G_1(.,y)\phi(y)\d y$ defines a bounded
operator from $L^p$ to $L^p$. 
\end{prop}

\begin{proof} 
We start with a smooth spinor $\phi$ compactly supported in $
B_{2R}(0)\times N\subset M$. For such a $\phi$ the spinor $P_1\phi$ is supported in
$B_{3R}(0) \times N\subset M$. We embed $B_{3R}(0)$ isometrically into a closed
Riemannian manifold $M_R$. Let $M_R \times N$. The metric on $M_R$ can
be chosen such that $D^{M_R \times N}-\mu$ is invertible, cf. Proposition~\ref{app_C}. The norm of $(D^{M_R \times N}-\mu)^{-1}\colon  L^p
\to L^p$ is denoted by $C_R(p)$.\\

For $p<\infty$ we estimate

\begin{align*}
\int_{M} |P_1\phi|^p\,\d x&= \int_{M} \left| \int_M G_1(x,y)\phi(y)\d y\right|^p
\d x \\
&\leq \int_{ B_{3R}(p_0)\times N} \left| \int_{B_{2R}(p_0)\times N} G(x,y)\phi(y)\, \d
y\right|^p\d x \leq \int_{M_R \times N} |(D^{M_R \times N}-\mu)^{-1}\phi|^p\, \d x\\
& \leq C_R(p)^p \int_{M_R \times N} |\phi|^p\, \d x =C_R(p)^p\,  \Vert
\phi\Vert_{L^p}^p.
\end{align*}

Next we want to consider arbitrary $\phi\in L^p(M,\Si_M)$, $p<\infty$. Then
$C_c^\infty(M,\Si_M)$ is dense in $L^p(M,\Si_M)$, and it suffices to consider
$\phi\in C_c^\infty(M,\Si_M)$. 
Choose points $(x_i)_{i\in I}\subset \mH_c^{k+1}$ as in Lemma~\ref{cover}. Then
$(B_{2R}(x_i)\times N)_{i \in I}$ and $(B_{3R}(x_i)\times N)_{i \in I}$ both
cover $\Mc^{m,k}$ uniformly locally finite. We denote the multiplicity of the second
cover by $L$ and choose a partition of unity $\eta_{i}$ subordinated to
$(B_{2R}(x_i)\times N)_{i\in I}$.

Let $\phi=\sum \phi_{i}$ where $\phi_{i}=\eta_{i}\phi\in C_c^\infty(B_{2R}(x_i) \times N,\Si_M)$.
Hence, $P_1\phi_i\in C_c^\infty (B_{3R} (x_i)\times N,
\Si_M)$.
Moreover, let $\bar{f}_i\colon  M\to M$ be given by $\bar{f}_i=(\Id, f_i)$ where $f_i$
is an isometry of $\mH_c^{k+1}$ that maps $x_i$ to $p_0$. We choose a lift of
$\tilde{f}_i$ to an isometry on the spinor bundle.  Due to the homogeneity of $\mH_c^{k+1}$ we have
$P_1(\tilde{f}_i\circ \phi\circ \bar{f}_i^{-1})=\tilde{f}_i\circ (P_1\phi)\circ \bar{f}_i^{-1}$.

Then, by triangle inequality and H\"older inequality and since for fixed $x$  the value
$P_1\phi_i(x)$ is nonzero for at most $L$ spinors $\phi_i$, we have
\[|P_1\phi (x)|^p= \left|\sum_{i}  P_1 \phi_i (x)\right|^p\leq  \left|\sum_{i} 
|P_1 \phi_i (x)|\right|^p\leq L^{p-1} \sum_i |P_1 \phi_i(x)|^p.\] Thus, we
obtain 
\begin{align*}
\Vert P_1\phi\Vert_{L^p(M)}^p &  \leq L^{p-1} \sum_{i} \Vert P_1 \phi_i
\Vert_{L^p(B_{3R}(x_i)\times N)}^p= L^{p-1}\sum_{i} \Vert P_1(\tilde{f}_i\circ \phi_i\circ \bar{f}_i^{-1})
\Vert_{L^p(B_{3R}(p_0)\times N)}^p\\
 &\leq L^{p-1} C_R(p)^p \sum_i \Vert
\tilde{f}_i\circ \phi_i\circ \bar{f}_i^{-1}\Vert_{L^p(B_{3R}(p_0)\times N)}^p\\
&=L^{p-1} C_R(p)^p \sum_i \Vert \phi_i\Vert_{L^p(B_{3R}(x_i)\times N)}^p\\
 &\leq
L^{p} C_R(p)^p\Vert
\phi\Vert_{L^p(M)}^p.\end{align*}
 
 It remains the case $p=\infty$. Let $\eta_i$ as above, and let $\phi\in L^\infty$. We decompose again $\phi=\sum \phi_i$
where $\phi_i=\eta_i\phi$ is compactly supported. Then, we obtain as
above that
 \begin{align*}
  \Vert P_1\phi\Vert_{L^\infty(M)}\leq \sum_i \Vert
P_1\phi_i\Vert_{L^\infty(B_{3R}(x_i)\times N)}\leq C \sum_i \Vert
\phi_i\Vert_{L^\infty(B_{3R}(x_i)\times N)}\leq C L \Vert
\phi\Vert_{L^\infty(M)}.
 \end{align*}
 \end{proof}
We now turn to the off-diagonal part $G_2$.\\
 
 Note that $\mH_c^{k+1}$ is homogeneous for all $c$. In particular, the
representation of the metric in polar coordinates -- $\d
r^2+\sinh_c^2(r)\,\sigma^k$ (cf. Section~\ref{sec_H}) -- is independent of the chosen origin of the
polar coordinates on $\mH_c^{k+1}$. We set $M_y(r):= \{x\in \Mc^{m,k}\ |\
\dist(x, N^y)
=r\}$ where $N^y = \{y_1\}\times N$  where $y=(y_1,y_2)\in \mH_c^{k+1}\times N$. Then, the volume
 $\vol (M_y(r))=f(r)^k\vol(N)\vol(S^k)=\sinh_c^k(r)\vol(N)\vol(S^k)$ is
independent of $y$. We will subsequently leave out the $y$ in the notation
and write $\vol (M(r))$.
 
% %%%%%%%%%%%%%%%%%%%%%%%%%%%%%%%%%%%%%%%%%%%%%
\begin{prop}\label{Green_smooth}  Using the notations from above, 
assume that there are constants $C, \rho>0$ with 
\[ \int_{M_y(r)} |G_2(x,y)|^2\d x\leq C e^{-2\rho r}\quad \text{for all\ }
r>0.\]
Let $p=1$ and $p=\infty$. Then, for $\rho >\frac{ck}{2}$ the operator $P_2\colon  \phi
\mapsto \int_M G_2(.,y)\phi(y)\d y$ from $L^p$ to $L^p$ is bounded.  
\end{prop}
 
\begin{proof} 
   We start with $p=1$ and estimate for $\phi\in C_c^\infty(M, \Si_M)$

          \begin{align*}
 \int_M |(P_2\phi)(x)|\, \d x&\leq  \int_M  \int_M |G_2(x,y)|
|\phi(y)|\,\d y\, \d x=  \int_M  \left( \int_M |G_2(x,y)| \d x\right) |\phi(y)|\, \d y
\\
 &=  \int_M  \left( \int_{\mR_+}  \int_{M_y(r)} |G_2(x,y)|\, \d  \tilde{x}\, \d r\right) |\phi(y)|\d y\\
 &\leq  \int_M  \left( \int_{\mR_+}  \vol(M(r))^\frac{1}{2}
\left(\int_{M_y(r)} |G_2(x,y)|^2 \, \d \tilde{x}\right)^\frac{1}{2} \right) |\phi(y)|\, \d y\\
 &\leq  C'  \int_{r\geq r_0}  \sinh_c^\frac{k}{2}(r) e^{-\rho r}\, \d r \Vert
\phi\Vert_{L^1}.
  \end{align*}
  where $\tilde{x}$ is the angular part and $r$ the radial part of $x$.

For $\rho>\frac{ck}{2}$ the integral $\int_{r\geq r_0}  \sinh_c^\frac{k}{2}(r)
e^{-\rho r}\,\d r $ is bounded. Hence, $P_2\colon  L^1\to L^1$ is invertible.

Next, we consider the other case $p=\infty$. Then for $\phi\in L^\infty(M, \Si_M)$

\begin{align*}
|(P_2\phi)(x)|&\leq  \int_{\frac{R}{2}}^\infty  \int_{M_x(r)}|G_2(x,y)|
|\phi(y)| \, \d
\tilde{y}  \,\d r\\
&\leq \int_{\frac{R}{2}}^\infty  \sup_{M_x(r)} |\phi| 
\left(\int_{M_x(r)}\hspace{-0.3cm} |G_2(x,y)|\, \d
\tilde{y}  \right)\,\d r \\
 &\leq   \Vert \phi \Vert_{L^{\infty}} \int_{\frac{R}{2}}^\infty   \Vert
G_2(x,y)\Vert_{L^2(M_x(r))} \vol
(M(r))^{\frac{1}{2}} \,\d r \\
 &\leq C \Vert \phi \Vert_{L^{\infty}} \int_{\frac{R}{2}}^\infty  e^{-\rho
r}  \sinh_c^{\frac{k}{2}}(r)\,\d r \leq  \tilde{C} \Vert \phi\Vert_\infty.
 \end{align*}
where for $\rho> \frac{ck}{2}$ the last inequality follows as above.
Thus, $\Vert P_2\phi\Vert_\infty \leq  \tilde{C} \Vert \phi\Vert_\infty$.
\end{proof} 

\section{ $\sigma_p$ contains the $L^p$-spectrum on
$\Mc^{m,k}$}\label{sec.first}

In this section we prove one direction of Theorem~\ref{main_inv}.

\begin{prop}\label{subset_spec_Hprod}  Let $p\in [1,\infty]$. Let $\lambda_0^2$,
$\lambda_0\geq 0$, be the lowest eigenvalue of the Dirac square on the closed
Riemannian spin manifold $N$.
 The $L^p$-spectrum of the Dirac operator on $\Mc^{m,k}$ is a subset of 
\[\sigma_p:= \left\{\mu\in \mC\ \Bigg|\  \mu^2=\lambda_0^2+\kappa^2, |\Im
\kappa|\leq ck\left| \frac{1}{p}-\frac{1}{2}\right| \right\}.\]
\end{prop}

\begin{proof}
We will show that $D-\mu\colon H_1^p\subset L^p\to L^p$ has a bounded inverse 
for all $\mu\in \mC\setminus \sigma_p$. Fix $\mu\in \mC\setminus \sigma_p$, and
let $\kappa\in \mC$ such that $\mu^2=\lambda_0^2+\kappa^2$. For $p=2$, the lemma follows from Remark \ref{specL2_MC}.

Let now $p\in \{ 1,\infty\}$ and $\mu \not\in \sigma_1=\sigma_\infty$. Then
$\mu\not\in \sigma_2$ and $(D-\mu)\colon  H_1^2(\Mc^{m,k}) \to L^2(\Mc^{m,k})$ has a bounded
inverse given by $P_\mu\colon  \phi \mapsto \int_{\Mc} G_\mu(x,y)\phi(y)dy$. By
Proposition \ref{Green_sing}, \ref{Green_smooth} and Lemma~\ref{decay_Green} the
operator $P_\mu\colon  L^p \to L^p$ is bounded for  $\left| \Im
\kappa\right|>ck\left|\frac{1}{p}-\frac{1}{2}\right|=c\frac{k}{2}$. Hence, the $L^1$- and the $L^\infty$-spectrum of
$D$ on $\Mc^{m,k}$ has to be contained in $\sigma_1=\sigma_\infty$. 

First we deal with the case that $\Im \kappa>0$. For $p\in [1,2]$ we use the Stein Interpolation Theorem~\ref{stein_int}:  Fix
$\epsilon >0$ and $y_0\in \R$. We set
$h(z):=\mu(z)^2:=\lambda_0^2+\kappa(z)^2:=\lambda_0^2+(y_0+\frac{ck}{2}\i z+\i
\epsilon)^2$  and
$A_z=(D^2-h(z))^{-1}$. By Remark~\ref{specL2_MC} the operators
\[A_{w+\i
y}=\left(D^2-\left[\lambda_0^2+\Bigg(y_0-\frac{ck}{2}y+\i\Big(\underbrace{\frac{
ck}{2}w+\epsilon}_{=\Im \kappa(z) >0}\Big)\Bigg)^2\right]\right)^{-1},\] for $0\leq
w\leq 1$ and $y\in \mR$, are bounded as operators from $L^2$ to $L^2$.  Furthermore \[A_{1+\i
y}=\left(D^2-\left(\lambda_0^2+\left(y_0-\frac{ck}{2}y+\i\left(\frac{ck}{2}
+\epsilon\right)\right)^2\right)\right)^{-1}\] is bounded from $L^1$ to $L^1$ as
seen above. Thus -- as required to apply the Stein
interpolation theorem -- $A_{\i y}$ and $A_{1+\i y}$ are bounded operators from
$L^1 \cap L^2$ to $L^1+L^2$.

 Let now $\phi\in L^1\cap L^2$ and $\psi\in L^\infty\cap L^2$. Set $S:=\{z\in
\mC\ |\
0\leq \Re z\leq 1\}$. We define $b_{\phi,\psi}(z)=\< A_z\phi,\psi\>$. The map $b_{\phi,\psi}$ is
analytic in the interior of $S$, since the resolvent is, see
Lemma~\ref{resolv_analytic}. Moreover, $|b_{\phi,\psi}(z)|\leq \Vert A_z\Vert
\Vert\phi\Vert_{L^2}\Vert\psi\Vert_{L^2} \leq (\max_{0\leq \Re z\leq 1} \Vert
A_z\Vert) \Vert\phi\Vert_{L^2}\Vert\psi\Vert_{L^2}$ where $\Vert A_z\Vert$
denotes the operator norm for $A_z\colon L^2\to L^2$. Thus, $b_{\phi,\psi}(z)$ is
uniformly bounded and continuous on $S:=\{z\in \mC\ |\ 0\leq \Re z\leq 1\}$. 
Thus, we can apply Theorem~\ref{stein_int} and obtain for $t\in (0,1)$ and $p=
\frac{2}{1+t}$ that
$A_t=\left(D^2-h\left(\frac{2}{p}-1\right)\right)^{-1}=\left(D^2-\left(\lambda_0^2+\left(y_0+ck\i\left(\frac{1}{p}
-\frac{1}{2}\right)+\i\epsilon\right)^2\right)\right)^{-1}$ is bounded from $L^p$ to $L^p$.

In the case ${\rm Im} \kappa<0$ we set analogously $A_z=(D^2-g(z))^{-1}$ for
$g(z)=\lambda_0^2+\left(y_0-\frac{ck}{2}\i z-\i\epsilon\right)^2$ and obtain that
$A_t=\left(D^2-g(\frac{2}{p}-1)\right)^{-1}$ is bounded from $L^p$ to $L^p$.
Since $y_0\in \R$ and $\epsilon>0$ can be chosen arbitrarily, we get for all
$\mu\in \mC\setminus \sigma_p$ that $\mu^2$ is not in the $L^p$-spectrum of
$D^2$. Using Lemma~\ref{spec_D2} the 
claim follows for $p\in [1,2]$ and with Lemma~\ref{conj_spectrum}.(i) for $p\in
[2,\infty)$.
\end{proof}

%%%%%%%%%%%%%%%%%%%%%%%%%%%%%%%%%%%%%%%%%%%%%%%%%%%%%%%%%%%%%%%
\section{Construction of test spinors on $\mH^{k+1}$}\label{sec.const}
%%%%%%%%%%%%%%%%%%%%%%%%%%%%%%%%%%%%%%%%%%%%%%%%%%%%%%%%%%%%%%%

In this section
we determine the Dirac $L^p$-spectrum of the hyperbolic space. The general case for $\Mc$ is given in the next section. 

\begin{prop}\label{spec_H} Let $p\in [1,\infty]$. The $L^p$-spectrum of the
Dirac operator $D$ on the  hyperbolic space $\mH^{k+1}$  is given by the set 
 \[\sigma_p^{\mH}:=\left\{\mu\in \mC\ \Bigg|\ |\Im \mu|\leq k\left|
\frac{1}{p}-\frac{1}{2}\right| \right\}.\]
\end{prop}

\begin{proof} From Proposition \ref{subset_spec_Hprod} we know that the
$L^p$-spectrum is contained in $\sigma_p^{\mH}$. Thus, it remains to show that
each element $\mu$ of $\sigma_p^{\mH}$ is contained in the $L^p$-spectrum of $D$. For
that we start with a similar ansatz as was used in
\cite[Lemma~7]{davies_simon_taylor_88} for the Laplacian.

Let  the hyperbolic space $\mH^{k+1}$ be modelled by 
the space $\{(y,x_1,\ldots,x_k)\ |\ y>0\}$ equipped with the 
metric $g=y^{-2}(dx_1^2 + \ldots + dx_k^2+dy^2)$. 
We set $e_i= y \frac{\partial}{\partial x_i}=y\partial_i$ for $i=1,\ldots, k$ 
and $e_y=y\frac{\partial}{\partial y}=y\partial_y$. Then,
$(e_y,e_1,\ldots,e_k)$ 
forms an orthonormal basis, which can assumed to be positively oriented. 
Further we have 
$[e_y,e_i]=e_i=-[e_i,e_y]$. All other commutators vanish. Then,
$-\Gamma_{iy}^i=\Gamma_{ii}^y=1$ and all other Christoffel symbols vanish.
The orthonormal frame $(e_y,e_1,\ldots,e_k)$ can be lifted to the spin structure
$\theta\colon \PSpin(\mH^{k+1})\to \PSO(\mH^{k+1})$, namely we choose a map
$E\colon \mH^{k+1}\to \PSpin(\mH^{k+1})$ with
$\theta(E)=(e_y, e_1,\ldots,e_k)$. A spinor is by definition a section of 
the associated bundle $\Si_{\mH^{k+1}}=
\PSpin(\mH^{k+1})\times_{\rho_{k+1}}\Si_{k+1}$, 
so every spinor can be written as $x\mapsto [E(x),\phi(x)]$ for a function
$\phi\colon \mH^{k+1}\to \Si_{k+1}$.

Hence, identifying $(e_y,e_1,\ldots,e_k)$ with the standard basis of $\mR^{k+1}$
we obtain \cite[(4.8)]{CGLS}, \cite[Lemma~4.1]{baer_diss}
\[ \nabla_{e_i}[E,\phi]=[E,\partial_{e_i}\phi +\frac{1}{2}e_i\cdot e_y\cdot
\phi];\quad  \nabla_{e_y}[E,\phi] = [E,\partial_{e_y} \phi] \]
and 
\begin{align}
 D[E,\phi]&=[E,\sum_{i=1}^k e_i\cdot \partial_{e_i} \phi +  e_y\cdot
\partial_{e_y} \phi - \frac{k}{2} e_y\cdot \phi]\nonumber\\
&=[E,\sum_{i=1}^k y e_i\cdot \partial_i \phi +  y e_y\cdot \partial_y \phi -
\frac{k}{2} e_y\cdot \phi].\label{D_hyp_eq}
\end{align}

Let $\psi_0\in\Si_{k+1}$ be a unit-length 
eigenvector of the Clifford multiplication with the vector $e_y=(1,0,\ldots,0)^t\in
\mR^{k+1}$ 
to the eigenvalue $\pm \i$, i.e.
 $e_y\cdot \psi_0=\pm \i\psi_0$.  
Set $\phi_n(x,y)=b(x)c_n(\log y) y^{\alpha} \psi_0$ where $\alpha\in \mC$, $b(x)$ is any
compactly supported function on $\R^{k}$, and where $c_n\colon  \R\to \R$ is chosen to be a
smooth cut-off function compactly supported on $(-4n, -n)$, $c_n|_{[-3n,-2n]}\equiv 1$
and $|c_n'|\leq 2/n$. Then for $p\in [1,\infty)$ one estimates $\Vert
c_n'\Vert_p^p/\Vert c_n\Vert_p^p\leq C n^{-p}\to 0$ as $n\to \infty$. For
$p=\infty$ we have $\Vert c_n'\Vert_\infty/\Vert c_n\Vert_\infty\leq 2/n \to 0$
as $n\to \infty$. Then we set $\Phi_n:=[E,\phi_n]$ and obtain
\begin{align}(D-\mu)\Phi_n=& \left[E, y  c_n(\log y)\,y^{\alpha}
\sum_{i=1}^k(\partial_i b)\, e_i\cdot  \psi_0 \pm b(x) c_n'(\log y) y^{\alpha}\i
 \psi_0\right.\nonumber\\
&\left. + b(x)c_n(\log y) (\pm \i\alpha \mp \i\frac{k}{2} - \mu) y^\alpha 
\psi_0\right].\label{D-mu} \end{align}

In the following we will use the notation $(X\cdot\,.\,)\in \End(\Si_{k+1})$ for
the Clifford multiplication by $X\in \mR^{k+1}$, and obviously its operator norm
$|(X\cdot\,.\,)|$ equals to the usual norm of $X$.

Let $\mu = s \pm \i k\left(\frac{1}{p} - \frac{1}{2}\right)$, $s\in \R$.
We choose $z=\log y$ and $\alpha=\frac{k}{2}\mp  \i \mu=\frac{k}{p}\mp \i s$. Thus, the
last summand in \eqref{D-mu}  vanishes and $p\Re\alpha=k$. Then, for $p\in
[1,\infty)$ we have
\begin{align*} &\frac{\Vert (D-\mu) \Phi_n\Vert_p}{\Vert \Phi_n\Vert_p}\\
&\leq \frac{\left( \int_{\R^{k}} |\sum _i  (\partial_{i} b) (e_i\cdot \,.\,)|^p
\int_0^\infty |c_n(\log y)|^py^{p \Re\alpha + p -
k-1}\right)^\frac{1}{p}}{\left(\int_{\R^{k}} |b(x)|^p \int_0^\infty |c_n(\log
y)|^p y^{p \Re \alpha - k-1}\right)^\frac{1}{p}}\\
&\phantom{\leq}+\frac{ \left( \int_{\R^{k}} |b(x)|^p \int_0^\infty |c_n'(\log y)|^p y^{p \Re
\alpha - k-1}\right)^\frac{1}{p}}{\left(\int_{\R^{k}} |b(x)|^p \int_0^\infty
|c_n(\log y)|^p y^{p\Re \alpha - k-1}\right)^\frac{1}{p}}\\
&\stackrel{p\Re\alpha=k}{\leq} c\left( \frac{\int_{\R^{k}} \sum_i |\partial_{i} b|^p
\int_0^\infty |c_n(\log y)|^py^{p-1}   }{\int_{\R^{k}} |b(x)|^p \int_0^\infty
|c_n(\log y)|^p y^{-1}}\right)^\frac{1}{p} + \left(\frac{\int_0^\infty
|c_n'(\log y)|^p y^{-1}}{\int_0^\infty |c_n(\log y)|^p
y^{-1}}\right)^\frac{1}{p}\\
&\stackrel{z=\log y}{\leq}  c\left(\frac{\int_{\R^{k}} \sum_i |\partial_{i} b|^p
\int_{-\infty}^0 |c_n(z)|^p e^{zp}   }{\int_{\R^{k}} |b(x)|^q \int_{-\infty}^0
|c_n(z)|^p }\right)^\frac{1}{p} +\left( \frac{\int_{-\infty}^0 |c_n'(
z)|^p}{\int_{-\infty}^0 |c_n(z)|^p }\right)^\frac{1}{p}\\
&\leq  C e^{-n} + \left(\frac{\int_{-\infty}^0 |c_n'(
z)|^p}{\int_{-\infty}^0 |c_n(z)|^p }\right)^\frac{1}{p}\to 0
\end{align*}
where the last inequality uses 
$$\int_{-\infty}^0 |c_n(z)|^p e^{zp}\, \d z
=\int_{-4n}^{-n} |c_n(z)|^p e^{zp}\, \d z\leq e^{-np}\int_{-4n}^{-n} |c_n(z)|^p \, \d z =e^{-np}\int_{-\infty}^{-0} |c_n(z)|^p \, \d z.$$

For $p=\infty$ we have $\mu=s\pm \i\frac{k}{2}$, $\alpha=\mp s$ and the estimate above is done analogously.

Summarizing, we have shown that $\partial \sigma_p^{\mH}$, the boundary of 
$\sigma_p^{\mH}$, is a subset of the Dirac $L^p$-spectrum for $\mH^{k+1}$ for
$p\in [1,\infty]$.
Note that $\sigma_s^{\mH}=\bigcup_{2\geq r\geq s}\pa \sigma_r^{\mH}$ for $s<2$
and $\sigma_s^{\mH}=\bigcup_{2\leq r\leq s}\pa \sigma_r^{\mH}$ for $s>2$,
respectively.
Thus, using the Riesz-Thorin interpolation theorem we see that $\sigma_p^{\mH}$  is a subset
of the $L^p$-spectrum of $D$ on $\mH^{k+1}$ for $p\in [1,\infty]$. 
\end{proof}

\begin{remark}\label{D_hyp_eq_2}
 From \eqref{D_hyp_eq} we obtain 
 \begin{align*}
  D^2 [E,\phi]=&[E, \sum_{i,j} y^2 e_i\cdot e_j\cdot \partial_{i}\partial_j \phi + \sum_i y^2 e_i\cdot e_y\cdot \partial_i\partial_y \phi-y\frac{k}{2} \sum_i e_i\cdot e_y\cdot \partial_i \phi\\
  & +  \sum_{i} y^2 e_y\cdot e_i\cdot \partial_{y}\partial_i \phi
  +  \sum_i y e_y\cdot e_i\cdot \partial_i \phi
  -y^2\partial_y\partial_y\phi-y\partial_y\phi+y\frac{k}{2}\partial_y\phi\\
  &-y\frac{k}{2} \sum_i e_y\cdot e_i\cdot \partial_i \phi+y\frac{k}{2}\partial_y\phi -\frac{k^2}{4}\phi]\\
  =&[E, -y^2\sum_i \partial_i^2 \phi-y^2 \partial_y^2 \phi + y(k-1)\partial_y\phi +  \sum_i y e_y\cdot e_i\cdot \partial_i \phi -\frac{k^2}{4}\phi].
 \end{align*}
 
 We use $\mu$ and $\phi_n=b(x)c_n(\log y) y^\alpha \psi_0$ of the last proposition with $b$, $\alpha$, $c_n$ and $\psi_0$ as therein. For $c_n$ we require additionally $|c_n''|\leq 8n^{-2}$. Hence, $\Vert c_n''\Vert_p/\Vert c_n\Vert_p\to 0$ as $n\to \infty$ for $p\in [1,\infty]$.  Then we have
\begin{align*}
  (D^2-\mu^2) [E,\phi_n]= &\left[E, \left( -y^2c_n(\log y) y^\alpha \sum_i \partial_i^2 b -y^2b \partial_y^2 (c_n(\log y) y^\alpha) + y(k-1)b\partial_y(c_n(\log y) y^\alpha)\right.\right. \\
  &\left.\left.-(\frac{k^2}{4}+\mu^2)bc_n(\log y) y^\alpha\right)\psi_0-\i c_n(\log y)y^\alpha \sum_i y (\partial_i b) e_i\cdot\psi_0 \right]\\
  = &\left[E,  -y^2c_n(\log y) y^\alpha \sum_i \partial_i^2 b \psi_0 -\i c_n(\log y)y^\alpha \sum_i y (\partial_i b) e_i\cdot\psi_0\right. \\
  &\left.  -y^{\alpha}b \left( c_n'' +(2\alpha+k-2)c_n'  + c_n\left(\alpha(\alpha-1) -(k-1)\alpha +\frac{k^2}{4}+\mu^2\right)\right)\psi_0\right]\\
   = &[E,  -c_n(\log y) y^{\alpha+2} \sum_i \partial_i^2 b\, \psi_0 -\i c_n(\log y)y^{\alpha+1} \sum_i (\partial_i b) e_i\cdot\psi_0\\
   &\left.-y^{\alpha}b \left( c_n'' +(2\alpha+k-2)c_n'\right)\psi_0\right],
 \end{align*}
 and by analogous estimates as in Proposition~\ref{spec_H} we have $\Vert (D^2-\mu^2) [E,\phi_n]\Vert_p/\Vert [E, \phi_n]\Vert_p\to 0$ as $n\to \infty$.
\end{remark}

\begin{remark} Note that while the $L^2$-spectrum of the hyperbolic space only consists of continuous spectrum, this is no longer true for the $L^p$-spectrum for $p\neq 2$ as can be seen by considering $0\in \sigma_p^H$:\\
We view the hyperbolic space $(\mH^{k+1},g_{\mH})$ modelled on the unit ball $B_1(0)\subset \mR^{k+1}$ of the Euclidean space and equipped with the metric $g_{\mH}= f^2g_E$ where $f(x)=\frac{2}{1-|x|^2}$ and $|.|$ denotes the Euclidean norm. 
Take a constant spinor $\psi$ on $B_1(0)$ normalized such that $\Vert \psi\Vert_{L^p(B_1(0), g_E)}=1$. Then $D^{g_E}\psi=0$. 
Using the identification of spinors of conformal metrics set $\phi:=f^{-\frac{k}{2}} \psi$. Then $D^{g_{\mH}}\phi=0$ and
$\Vert \phi\Vert_{L^p(g_{\mH})}^p=\int_{B_1(0)} f^{k+1-\frac{k}{2}p} |\psi|^p \vo_{g_E}$. Thus,  $\phi$ is an $L^p$-harmonic spinor if and only if $ \int_{B_1(0)} (1-|x|^2)^{-k-1+\frac{k}{2}p}\vo_{g_E}<\infty$, i.e., if and only if $ \int_{0}^1 (1-r^2)^{-k-1+\frac{k}{2}p}r^{k}\d r<\infty$. This is true precisely if $p>2$. Thus, for all $p>2$ the $L^p$-kernel of the Dirac operator on $(\mH^{k+1}, g_{\mH})$ is nontrivial.
\end{remark}

%%%%%%%%%%%%%%%%%%%%%%%%%%%%%%%%%%%%%%%%%%%%%%%%%%%%%%%
\section{The $L^p$-spectrum on $\Mc^{m,k}$ contains $\sigma_p$}\label{sec.last}
%%%%%%%%%%%%%%%%%%%%%%%%%%%%%%%%%%%%%%%%%%%%%%%%%%%%%%%

In this section we complete the proof of Theorem~\ref{main_inv}. In Proposition
\ref{subset_spec_Hprod} it was shown that the $L^p$-spectrum on $\Mc^{m,k}$
is contained in $\sigma_p$. Thus, the converse remains to be shown. The case $N=\{y\}$ was solved in Proposition \ref{spec_H}.

Recall that by Lemma~\ref{spec_or} and Example~\ref{spec_or_ex} the Dirac $L^p$-spectrum on $\Mc^{m,k}$ is point symmetric, i.e., it is symmetric with respect to the reflection $\lambda\mapsto -\lambda$.

Let now $\mu\in \partial \sigma_p$ with  $\mu^2=\lambda_0^2+\kappa^2$, $|\Im \kappa|=
ck\left| \frac{1}{p}-\frac{1}{2}\right| $ be given. By Proposition \ref{spec_H}
and scaling, we see that $\kappa$ is in the spectrum of the Dirac operator of $\mH_c^{k+1}$.
Then, by Lemma~\ref{spec_D2}
$\kappa^2$ is in the $L^p$-spectrum of $(D^{\mH_c^{k+1}})^2$, and by Remark~\ref{D_hyp_eq_2} there is a
sequence $\psi_i\in \Gamma(\Si_{\mH_c^{k+1}})$ with $\Vert
((D^{\mH_c^{k+1}})^2-\kappa^2)\psi_i\Vert_{L^p(\mH_c^{k+1})}\to 0$ while $\Vert
\psi_i\Vert_{L^q(\mH_c^{k+1})}=1$. Moreover, by Remark \ref{spec_closed} there is a
$\psi\in \Gamma(\Si_N)$ with $\Vert \psi\Vert_{L^q(N)}=1$ and
$(D^N)^2\psi=\lambda_0^2\psi$.

Assume that at least one of the dimensions of $N$ and $\mH_c^{k+1}$ is even.
Then $\Si_{\Mc}=\Si_{\mH_c^{k+1}} \otimes \Si_N$ and by \eqref{Dirac_square_prod}
we have $D^2=({D}^{\mH_c^{k+1}})^2+({D}^N)^2$. We set $\phi_i=\psi_i\otimes\psi$. Then
\begin{align*} \Vert (D^2-\mu^2)\phi_i\Vert_p=& \Vert \psi_i\otimes ((D^N)^2-\lambda_0^2)\psi +
((D^{\mH_c^{k+1}})^2-\kappa^2) \psi_i\otimes \psi\Vert_p\\
=&\Vert((D^{\mH_c^{k+1}})^2-\kappa^2) \psi_i\otimes \psi\Vert_p\to 0.\end{align*}

Thus, $\mu^2$ is in the $L^p$-spectrum of $D^2$. By the point symmetry of the spectrum and by Lemma~\ref{spec_D2} both $\mu$ and $-\mu$ are in the $L^p$-spectrum of $D$.

Similarly we obtain the result if both the dimensions of $N$ and $\mH_c^{k+1}$
are odd by setting
$\phi_i:=\psi_i\otimes (\psi, \psi)$ in notation of Section~\ref{spin_prod}. 

Up to now we have shown that all $\mu\in \partial \sigma_p$ are in the $L^p$-spectrum of the Dirac operator on $\Mc$. Following the same arguments as in the last lines of the proof of Proposition~\ref{spec_H} the proof of Theorem~\ref{main_inv} is completed.

\begin{remark}\label{rem.spec.D2}
From Theorem \ref{main_inv} and Lemma \ref{spec_D2}, we can immediately read of
the $L^p$-spectrum of $D^2$ on $\Mc^{m,k}$. This consists of the closed
parabolic region bounded by 
\[ s\in \mR \mapsto \lambda_0^2-c^2k^2\left(\frac{1}{p}-\frac{1}{2}\right)^2 +s^2
+2\i sck\left|\frac{1}{p}-\frac{1}{2}\right|.\]
Let us compare the $L^p$-spectrum for $D^2$ on  $\Mc^{k+1,k}=\mH^{k+1}$
($c=1$ and $\lambda_0=0$) 
\[ s\in \mR \mapsto -k^2\left(\frac{1}{p}-\frac{1}{2}\right)^2 +s^2 +2\i
sk\left(\frac{1}{2}-\frac{1}{p}\right),\]
 with the one of the Laplacian on functions whose $L^p$-spectrum is given by
the closed parabolic region
bounded by \cite[(1.5)]{davies_simon_taylor_88}
\[ s\in \mR \mapsto k^2\frac{1}{p}\left(1-\frac{1}{p}\right)+s^2 +2\i
sk\left(\frac{1}{2}-\frac{1}{p}\right).\]
Up to a shift in the real direction this is the same spectrum. However the qualitative difference is that for $p\neq 2$ the spectrum of $D^2$ contains negative real numbers, in contrast to the Laplacian.
\end{remark}

\begin{appendix}

%%%%%%%%%%%%%%%%%%%%%%%%%%%%%%%%%%%%%%%%%%%%%%%%%%%%%%%%%%%%%%%%%%%%%
%
% general notes on the L^p spectrum
%
%%%%%%%%%%%%%%%%%%%%%%%%%%%%%%%%%%%%%%%%%%%%%%%%%%%%%%%%%%%%%%%%%%%%%
\section{Function spaces}

We want to recall some analytical facts which are helpful to define spinorial
function spaces on manifolds.

Let $(M^n,g)$ be an $n$-dimensional Riemannian spin manifold with Dirac operator
$D$. A distributional spinor (or distribution with spinor values) is a linear map $C_c^\infty(M, \Si_M)\to \mC$ with the usual continuity properties of distributions. Any spinor with regularity $L^1_{loc}$ defines a distributional spinor by using the standard $L^2$-scalar product on spinors.

Then $D\phi$ can be defined in the sense of distributions.
Let $H_1^s(M, \Si_M)$ be the set of distributional spinors $\phi$,
such that $\phi$ und $D\phi$ are in $L^s$, $s\in [1,\infty]$. Equipped with the
norm 
$\|\phi\|_{H_1^s}:=\|\phi\|_s+ \|D\phi\|_s$ this is a Banach space. This norm is the graph norm of $D$ viewed as an operator in $L^s$ to $L^s$.

\begin{lemma}Let $1\leq s<\infty$. 
$C^\infty_c(M, \Si_M)$ is dense in $H_1^s(M, \Si_M)$.
\end{lemma}

\begin{proof}
Assume that $\phi\in H_1^s(M, \Si_M)$, $s<\infty$, is given.
For a given point $p\in M$ and for any $R>0$ one can find a 
compactly supported smooth function
$\eta_R\colon M\to [0,1]$ such that $\eta_R\equiv 1$ on $B_R(p)$ and such 
that $|\nabla\eta_R|\leq R^{-1}$.
Then one easily sees $\lim_{R\to \infty} \|\phi -\eta_R \phi\|_s=0$. Further
we calculate
 \[\|D(\phi -\eta_R \phi)\|_s\leq \|\nabla\eta_R\cdot\phi \|_s +  
\|(1-\eta_R)D\phi \|_s \to 0 \quad {\rm as\ }R\to\infty.\]
Thus the elements with compact support are dense in  $H_1^s(M, \Si_M)$.
Now if  $\psi\in H_1^s(M, \Si_M)$ has compact support, it follows from standard  
results that it can be approximated by smooth compactly supported spinors.
\end{proof}

Thus, for $s<\infty$, $H_1^s(M, \Si_M)$ is equal to the completion of
$C^\infty_c(M, \Si_M)$ 
with respect to the graph norm of $D\colon  L^s\to L^s$.

\begin{lemma} Let $1<s<\infty$.
On manifolds with bounded geometry, the $H_1^s$-norm is equivalent
to the norm $\Vert \phi\Vert_s+\Vert \nabla \phi\Vert_s$.
\end{lemma}

The proof of the lemma relies on local elliptic estimates which follow from 
the Calderon-Zygmund inequality, e.g.\ \cite[Theorem 9.9]{GT}, see also 
\cite[Lemma~3.2.2]{Ammha} for the geometric adaptation.

\section{General notes on the $L^p$-spectrum} \label{sec_Lp}
 
In this section we collect general facts on the $L^p$-spectrum of the Dirac
operator. Unless stated otherwise, we only assume that $(M,g)$ is complete.

 Let $D\colon  H_1^2(M, \Si_M)=\dom\, D\subset L^2(M, \Si_M) \to L^2(M, \Si_M)$ be the
classical 
Dirac operator on $L^2$-spinors. 
The set of compactly supported spinors $C_c^\infty(M, \Si_M)$ is a core
of $D$, i.e., $D$ is the closure of $D|_{C_c^\infty(M, \Si_M)}$ w.r.t. the graph
norm $H_1^2$. If we consider the restriction $D|_{C_c^\infty(M, \Si_M)}$ and
complete it w.r.t. the graph norm $\Vert \phi\Vert_{H_1^s}=\Vert
\phi\Vert_s+\Vert D \phi\Vert_s$ for $1\leq  s < \infty$, then we
obtain for each $s$ a closed Dirac operator $D_s\colon  H_1^s=\dom\, D_s \subset L^s
\to L^s$. 

For $s=\infty$ we define $D_\infty\colon  H_1^\infty  \to L^\infty$, $\psi\mapsto D_\infty \psi$ by $(D\phi, \psi)=(\phi, D_\infty \psi)$ for all $\phi\in C_c^\infty(M, \Si_M)$. Then $D_\infty$ is a closed, continuous extension of
$D|_{C_c^\infty(M, \Si_M)}$ but $C_c^\infty(M, \Si_M)$ is in general no longer a
core for this operator. Note that in contrast to that, in the standard
literature for $L^p$-theory of the Laplacian, e.g.
\cite{davies_simon_taylor_88}, the operator for $s=\infty$ is directly defined to be as the
adjoint operator for $s=1$. For $s<\infty$ one can define $D_s$ distributional as well and the resulting operator coincides with the definition given above as will be seen in Lemma~\ref{lem1}.

Next, we can examine the adjoint of the operator $D_s\colon L^s\to L^s$ with respect
to the duality pairing $\langle .,.\rangle\colon  L^s\times (L^s)^*\to \C$ whose
restriction to compactly supported spinors coincides with the hermitian
$L^2$-product. We use the convention that this pairing is antilinear in the
second
component. The adjoint $D_s^*$ is an operator in $(L^s)^*$. For $1\leq s<\infty$ and
$s^{-1}+(s^*)^{-1}=1$, $(L^s)^*=L^{s^*}$ whereas $(L^\infty)^*$ is
larger than $L^1$. From the formal
self-adjointness of $D$ we see, that $D_{s^*}|_{C_c^\infty(M, \Si_M)} =
D_s^*|_{C_c^\infty(M, \Si_M)}$. Moreover, we have

\begin{lem}\label{lem1}  
For all $\phi\in H_1^s$ and $\psi\in H_1^{s^*}$, $1\leq   s \leq \infty$,  we have
\[ (D_s\phi,\psi)=(\phi, D_{s^*}\psi).\] 
\end{lem}

\begin{proof}For $1<s< \infty$, let $\phi_i, \psi_j\in C_c^\infty(M, \Si_M)$  with $\phi_i\to \phi$
in $H_1^s$ and $\psi_j\to \psi$ in $H_1^{s^*}$. Then,
\begin{align*}
\int_M \langle D_s\phi,\psi\rangle \vo_g\leftarrow \int_M \langle
D_s\phi_i,\psi_j\rangle \vo_g = \int_M \langle \phi_i,D_{s^*}\psi_j\rangle
\vo_g\to
\int_M \langle \phi,D_{s^*}\psi\rangle \vo_g
\end{align*}
as $i,j\to \infty$.\\
Let now $s=1$. For $\phi\in C_c^\infty(M, \Si_M)$ the equality follows from the distributional definition of $D_\infty$. The rest follows since $C_c^\infty(M, \Si_M)$ is dense in $H_1^1$. The remaining case $s=\infty$ just follows from the last one by interchanging $s$ and $s^*$.
\end{proof}

\begin{lem} \label{lem2}   For all $1\leq s<\infty$ the operators $D_{s^*}$ and
$D_s^*$ coincide.
\end{lem}

\begin{proof}
For $\psi\in H_1^{s^*}$ Lemma~\ref{lem1} yields $(D_s\phi, \psi)=(\phi,
D_{s^*}\psi)$   for all $\phi\in H_1^s=\dom\, D_s$. This implies $\psi\in \dom\,
D_s^*$ and $D_s^*\psi=D_{s^*}\psi.$ Hence, $H_1^{s^*}\subset \dom\, D_s^*$ and
$D_s^*|_{H_1^{s^*}}=D_{s^*}\colon  H_1^{s^*}\subset L^{s^*}\to L^{s^*}$.
It remains to show that $\dom\, D_s^*\subset H_1^{s^*}$: Let $\psi\in
\dom\, D_s^*\subset (L^s)^*=L^{s^*}$. Then there is a $\rho\in
L^{s^*}$ such that for all $\phi\in \dom\, D_s$ it holds $(D_s\phi, \psi)=
(\phi, \rho)$. In particular, this is true for all $\phi\in C_c^\infty(M,
\Si_M)$. In other words $D_s^*\psi=\rho$ in the sense of distributions. Thus, $\psi\in H_1^{s^*}$.
\end{proof}

Since $\phi\in H_1^s\cap H_1^r$ implies $D_s\phi=D_r\phi$ we often denote all
those
Dirac operators in the following just by $D$.

Moreover, a closed operator $P\colon \dom\, P\subset V_1\to V_2$ 
between Banach spaces $V_i$, and with dense domain $\dom\, P$, 
will be called invertible if there exists a bounded inverse $P^{-1}\colon V_2\to V_1$.
We will use the phrase ``$P$ has a bounded inverse'' synonymously.

\begin{lem}\label{conj_spectrum}   
Let $1 \leq  s<\infty$. 
\begin{itemize}
\item[(i)] If $\overline{\mu}$ is in the $L^{s^*}-$spectrum of the Dirac
operator where $(s^*)^{-1}+s^{-1}=1$, then $\mu$ is in its $L^s-$spectrum.
\item[(ii)] Let $D_s-\mu$ be invertible. Then,
$(D_{s^*}-\bar{\mu})^{-1}=((D_s-\mu)^{-1})^*$ and
$\Vert (D_{s^*}-\bar{\mu})^{-1}\Vert=\Vert (D_s-\mu)^{-1}\Vert$.
\end{itemize}
\end{lem}

\begin{proof}We prove this for  $\mu=0$. For arbitrary $\mu$ this is
done analogously.\\
Assume that $0$ is not in the $L^s$-spectrum of $D$, i.e., it has a bounded
inverse $E=D^{-1}\colon L^s\to L^s$ with range $\ran\, E=H_1^s$.
Let $\phi\in L^{s^*}$. Since $E$ is bounded, $f\colon L^s \to \C, \rho\mapsto
(E\rho, \phi)$ is a bounded functional and, thus, $f$ is in the dual space of $L^s$, i.e.,  there is $\psi\in L^{s^*}$ with
$(\rho,\psi)=f(\rho)=(E\rho,\phi)$ for all $\rho\in L^s$. Hence, $\phi\in
\dom\, E^*$, i.e., $\dom\, E^*=L^{s^*}$. 

Now we can estimate for all $\phi\in H_1^s$ and all $\psi\in L^{s^*}$ that 
$(D\phi, E^*\psi)=(ED\phi,\psi)=(\phi, \psi)$ which implies $E^*\psi\in
\dom\, D^*$ and $D^*E^*\psi=\psi$.
Thus, $\ran\, D^*=L^{s^*}$ and $D^*E^*=\Id\colon L^{s^*}\to L^{s^*}$.

If $\rho\in L^s$ and $\phi\in \dom\, D^*$, we get
$(\rho, E^*D^*\phi)= (E\rho, D^*\phi)=(DE\rho,\phi)=(\rho,\phi)$. Hence,
$E^*D^*=\Id\colon \dom\, D^*\to \dom\, D^*$. Together with the corresponding statement from
above this gives that $(D^{-1})^*=(D^*)^{-1}$. Thus,
 $0$ is not in the $L^{s^*}$-spectrum of $D$. This proves (i) and the first claim of (ii).
% (ii) From (i) we know that $D_{s^*}$ is also invertible. Let $\phi\in
% L^{s^*}$, $\psi\in L^{s}$. Then there exist $\hat{\phi}\in H_1^{s^*}$ and
% $\hat{\psi}\in H_1^{s}$ with $D_{s}\hat{\psi}=\psi$ and
% $D_{s^*}\hat{\phi}=\phi$. Thus, we have using Lemma~\ref{lem1}
% \begin{align*}
% (D_{s^*}^{-1}\phi,\psi)&=(\hat{\phi}, D_s\hat{\psi})= (D_{s^*}\hat{\phi},
% \hat{\psi})=(\phi, D_{s}^{-1}\psi)
% \end{align*}
% which proves the first claim. 
The operator norm of an operator and its adjoint
coincide, see \cite[Thm VI.2]{RS}. Thus, the equality of the operator norms
follows.
 \end{proof}
 
\begin{cor}\label{inv_interpolation}  
If $D\colon H_1^q\to L^q$ has a bounded inverse for some $q\in
(1,\infty)$. Then as an operator from $H_1^s \to L^s$ it has a bounded inverse
for all $s\in [q_1, q_2]$  where $q_1=\min\{q,q^*\}$, $q_2=\max\{q,q^*\}$, and
 $(q^*)^{-1}+q^{-1}=1$. In particular, the
$L^2$-spectrum of $D$ is a subset of the $L^q$-spectrum.
\end{cor}

\begin{proof}
This Lemma follows directly from the Riesz-Thorin Interpolation
Theorem~\ref{RTint}
 (using $\mathcal{D}=C_c^\infty(M, \Si_M)$) and Lemma~\ref{conj_spectrum}.
\end{proof}

\begin{lemma}\label{resolv_analytic} 
Let $1\leq s\leq \infty$. Let $R_s=\mC\setminus \spec_{L^s}(D)$ be the resolvent
set of $D\colon  L^s\to L^s$. Then, the resolvent
 \[ \mu\in  R_s \mapsto (D-\mu)^{-1}\in \mathcal{B}(L^s) \] is analytic, i.e., 
the map is locally given by a convergent power series with coefficients in
$\mathcal{B}(L^s)$. Here, $\mathcal{B}(L^s)$ denotes the set of bounded
operators from $L^s$ to itself. 
\end{lemma}

% The Lemma is still true when replacing $L^s$ by Banach spaces $X$, $Y$ and $D$
% by any unbounded linear operator densely defined on $X$ and with values in $Y$.

\begin{proof} The proof is done similar as in the case of bounded operators \cite[Satz
23.4]{hirzebruch_scharlau}:  Choose $ \mu_0\in R_s$ and $\mu\in \mC$ such that $|\mu-\mu_0|<\Vert (D-\mu_0)^{-1}\Vert^{-1}$. Then, 
one calculates that $D-\mu$ is invertible as well, see the proof of \cite[Lemma
23.2]{hirzebruch_scharlau}. Here we used the fact the operator $(D-\mu)^{-1}$ and $(D-\mu_0)^{-1}$ have the common core $C_c^\infty(M, \Si_M)$. Then, 
\[ (D-\mu)^{-1}=\sum_{n=0}^\infty \left( (D-\mu_0)^{-1}\right)^{n+1}\left( \mu-\mu_0\right)^{n}.\]
\end{proof}

For rounding up our presentation we will next add a lemma not needed in our context but maybe helpful to other applications.
 \begin{lem}\label{inv-gap}
\begin{enumerate}
\item The operator $D\colon  H_1^s\subset L^s\to L^s$, $s\in [1,\infty]$, is an invertible map onto its
image 
if and only if there is a constant $C>0$ with $\Vert D\phi\Vert_{s}\geq C\Vert
\phi \Vert_s$ for all $\phi\in H_1^s$. 
\item Under the above conditions the image $D(H_1^s)$ is closed.
\item Let $s^{-1}+(s^*)^{-1}=1$, $s<\infty$, and assume the conditions from above. Then~$D$ is surjective if and 
only if there is a $C>0$ with 
$\Vert D\phi\Vert_{s^*}\geq C\Vert \phi \Vert_{s^*}$ for all $\phi\in
H_1^{s^*}$. 
\end{enumerate}
\end{lem}

\begin{proof}
(1) The proof is straightforward.\\
% First let $D$ be invertible onto its image. Then there exists a bounded 
% inverse $D^{-1}\colon  D(H_1^s)\to L^s$. Hence, there is a constant $C>0$ with 
% $\Vert D^{-1}\phi\Vert_s\leq C\Vert\phi\Vert_s$ for all $\phi\in D(H_1^s)$. 
% Inserting now $\phi=D\psi$ for $\psi\in H_1^s$ gives the desired inequality.
% Let now $C$ be a positive constant 
% with $\Vert D\phi\Vert_{s}\geq C\Vert \phi \Vert_s$ for all $\phi\in H_1^s$. 
% Thus, due to the linearity of  $D$, $D$ has to be injective. 
% Hence, there is an inverse $D^{-1}\colon D(H_1^s)\to L^s$. 
% Inserting $\phi= D^{-1}\psi$ in the assumed inequality gives 
% that $D^{-1}$ is also bounded. 
(2) The operator $D\colon H_{1}^s\to D(H_{1}^s)$, where the latter space is equipped with the $L^s$-norm, is a bijective bounded linear map. Hence, $D(H_{1}^s)$ is a complete subspace of $L^s$ and thus closed.\\
(3) Suppose that  $D(H_1^s)$ is a proper subspace of $L^s$. Due 
to Hahn-Banach there is a non-zero continuous functional 
$\psi\colon L^s\to \mC$ vanishing on 
$D(H_1^s)$. We interpret $\psi$ as an element in $L^{s^*}$ using the Riesz 
representation theorem, i.e.\ $\psi\in L^{s^*}$ is orthogonal on $D(H_1^s)$.
Then, $\psi\in \dom\, (D_{s})^*$, and we even have $D_{s}^*\psi=0$. Hence, by Lemma~\ref{lem2} $\psi\in
H_1^{s*}$.
%  By formally self-adjointness of $D$ and by density of $C_c^\infty$ in 
% $H_1^s$ 
% we see that $D\psi$ is orthogonal to $H_1^s$, and thus 
% $D\psi=0$.  
This contradicts the estimate.

Now assume that $D$ is surjective. 
Then there is a bounded operator $D^{-1}\colon L^s\to L^s$, inverse to $D$.
Thus $(D^{-1})^*\colon L^{s^*}\to L^{s^*}$ is bounded as well, and  $(D^{-1})^*$
is the inverse of $D^*\colon H_1^{s^*}\to L^{s^*}$. The fact that 
the latter map has a bounded inverse is equivalent to the existence of
a constant $C>0$ with $\Vert D\phi\Vert_{s^*}\geq C\Vert \phi \Vert_{s^*}$. 

\end{proof}

\begin{rem}\label{spec_closed} The $L^s$-spectrum of the Dirac operator $D$ on a
closed manifold $(M^m,g)$ is independent of $s$. We sketch the proof: Let
$\phi$ be an $L^2$-eigenvalue of $D$. Then regularity theory implies that
$\phi\in C^\infty(M, \Si_M)$ and, hence, $\phi\in L^s$ for all $1\leq s\leq
\infty$. In particular, $\Spec_{L^2}^M(D)\subset \Spec_{L^s}^M(D)$. Let now
$\mu\not\in \Spec_{L^2}^M(D)$, i.e., $(D-\mu)^{-1}\colon  L^2\to L^2$ is bounded. Let
$G(x,y)$ be the unique Green function of $D-\mu$, see
Proposition~\ref{Green-compact}. Then, $\int_M |G(.,y)|^2 dy$ is bounded
uniformly in $y$. H\"older's inequality implies that also $\int_M |G(.,y)| dy$
is bounded uniformly in $y$. Hence, $(D-\mu)^{-1}\colon  L^1\to L^1$ is a bounded
operator. Then interpolation gives that
$(D-\mu)^{-1}\colon  L^s\to L^s$ is bounded for all $1\leq s<2$. Because of $\Spec_{L^2}(D)\subset \R$ the same is true for $(D-\bar{\mu})^{-1}\colon  L^s\to L^s$, and by using Lemma~\ref{conj_spectrum} we get that
$(D-\mu)^{-1}\colon  L^s\to L^s$ is bounded for all $2< s<\infty$. It remains
$s=\infty$: Let $r>m$. Then by the Sobolev Embedding Theorem %~\ref{est_iv}
$H_1^r\hookrightarrow L^\infty$ is bounded. Moreover, by the 
discussion  above and using the fact that $H_1^r$ carries the graph norm of $D$ we know that $(D-\mu)^{-1}\colon  L^r\to H_1^r$ is bounded for $\mu\not\in \Spec_{L^2}^M (D)$ the H\"older inequality gives that 

\[ (D-\mu)^{-1}\colon \, L^\infty \to L^r \to H_1^r\to L^\infty\]
is bounded.
\end{rem}

\begin{lem}\label{spec_D2} Let $1\leq s\leq \infty$, and let $\Spec_{L^s}^M(D)\neq \mC$. 
Then the complex number $\mu^2$ is in the $L^s$-spectrum of $D^2$ if and only if
$\mu$ or $-\mu$ is in the $L^s$-spectrum of $D$.
\end{lem}

\begin{proof}
We start with the ``only if'' part. So assume that both $\mu$ and $-\mu$ are
not in the $L^s$-spectrum of~$D$. Then we have bounded operators
$(D-\mu)^{-1}\colon L^s\to L^s$ and $(D+\mu)^{-1}\colon L^p\to L^p$. It is then easy to
verify
that $(D-\mu)^{-1}\circ (D+\mu)^{-1}\colon L^s\to L^s$ is a bounded inverse of 
$D^2-\mu^2=(D+\mu)\circ (D-\mu)$. Thus $\mu^2$ is not in the $L^p$-spectrum of
$D^2$.

In order to prove the ``if'' statement, we assume that $\mu^2$ is not in 
the spectrum of $D^2$. Then $D^2-\mu^2$ has a bounded inverse 
$P:=(D^2-\mu^2)^{-1}\colon L^s\to L^s$. Let $\psi\in P(L^s)$. Then $\psi\in L^s$ and $D^2\psi\in L^s$. Next we will show that this implies $D\psi\in L^s$. For that we choose $\lambda\not\in \Spec_{L^s}^M(D)$. Then
$D\psi=(D-\lambda)^{-1}(D^2-\lambda^2)\psi-\lambda\psi$, and hence $D\psi\in L^s$. Thus, $P(L^s)\subset H_{1}^s$.
%Using the fact that the manifold has bounded 
%geometry, standard regularity theorems show that $P$ is actually a bounded
%operator $L^p\to \todo{H^{2,p}}$. 
Hence $Q_1:=(D\pm \mu)\circ P$ 
is a bounded operator with $\dom\, Q_1=L^s$, and one easily checks that this a
right inverse to $(D\mp \mu)$. Similarly, one shows that 
$Q_2:=P\circ (D\pm \mu)$ is a left inverse of  $(D\mp \mu)$. A priori $Q_2$
is only defined on $H_{1}^s$, but using $Q_1=Q_1\circ (D\mp \mu)\circ Q_2=Q_2$
it is clear 
that $Q_2$ and $Q_1$ coincide on $H_1^s$. So the integral kernels of 
$Q_1$ and $Q_2$ have to coincide, so $Q_1$ is a left and right inverse of
$(D\mp \mu)$ and thus $\pm\mu$ is not in the spectrum of $D$.
\end{proof}

\begin{remark}
 In the case $1<s<\infty$ and $M$ of bounded geometry, one can also prove that $\Spec_{L^s}^M(D)=\mC$ implies $\Spec_{L^s}^M(D^2)=\mC$: As in the proof of the ``if'' statement from above one has to show that $D\psi\in L^s$. This can be proven using regularity theory on manifolds of bounded geometry.
\end{remark}

\begin{lem}[Pointwise symmetries]\label{spec_ps} Let $1\leq s\leq \infty$. Let
$(M,g)$ be an $m$-dimensional Riemannian spin manifold.
\begin{itemize}
\item[(i)] $m\equiv 0\mod 2$: The number $\mu$ is in the $L^s$-spectrum of $D$ 
if and only if  $-\mu$ is in the $L^s$-spectrum of $D$ if and only if  
$\bar{\mu}$ is in the $L^s$-spectrum of $D$. 
\item[(ii)] $m\equiv 1\mod 4$: The number $\mu$ is in the $L^s$-spectrum of $D$ 
if and only if  
$-\bar{\mu}$ is in the $L^s$-spectrum of $D$.
\item[(iii)] $m\equiv 3\mod 4$: The number $\mu$ is in the $L^s$-spectrum of
$D$ 
if and only if  
$\bar{\mu}$ is in the $L^s$-spectrum of $D$.
\end{itemize}
\end{lem}

\begin{proof} By \cite[Prop. p. 31]{friedrich:00} we have a map $\alpha\colon 
\Sigma_m\to\Sigma_m$ that is 
\begin{itemize}
\item a $\Spin(m)$-equivariant real structure that  anticommutes with Clifford
multiplication if $m\equiv 0,1 \mod 8$.
\item a $\Spin(m)$-equivariant  quaternionic structure that  commutes with
Clifford multiplication if $m\equiv 2,3 \mod 8$.
\item a $\Spin(m)$-equivariant  quaternionic structure that  anticommutes with
Clifford multiplication if $m\equiv 4,5 \mod 8$.
\item a $\Spin(m)$-equivariant real structure that  commutes with Clifford
multiplication if $m\equiv 6,7 \mod 8$.
\end{itemize}
Note that by definition real structure means that $\alpha^2=\Id$ and
$\alpha({\rm i} v)=-{\rm i}\alpha(v)$. Moreover, quaternionic structure means
that $\alpha^2=-\Id$ and $\alpha({\rm i} v)=-{\rm i}\alpha(v)$.

 Due to the $\Spin(m)$-equivariance $\alpha$ induces a fiber preserving map $\tilde{\alpha}$ on
the spinor bundle with the same properties as above.
 Thus, \[(D-\mu)\circ \tilde{\alpha} (\phi)
             =\left \{\begin{matrix} 
 \tilde{\alpha}\circ (-D-\bar{\mu})(\phi) & m\equiv 0,1 \mod 4\\
 \tilde{\alpha}\circ (D-\bar{\mu}) (\phi) & m\equiv 2,3 \mod 4. 
                    \end{matrix}
 \right.\]
Thus, if $\mu$ is in the $L^s$-spectrum of $D$ then $-\bar{\mu}$ (resp.
$\bar{\mu}$) in the $L^s$-spectrum of $D$ for $m\equiv 0,1$ (resp. $2$,$3$)
$\mod$ 4.  This gives (ii) and (iii).

If $m$ is even, then $D (\omega_M\cdot \phi)=-\omega_M\cdot D\phi$. Thus,
the spectrum is symmetric when reflected on the imaginary axis. 
Together with the symmetries from above, (i) follows.
\end{proof}

\begin{lem}[Orientation reversing isometry]\label{spec_or} Let $1\leq s\leq \infty$.
Assume there is an orientation reversing isometry $f\colon M^m\to M^m$ that ``lifts''
to
the spin structure as described in the proof.
 Then  $\mu$ is in the $L^s$-spectrum of $D$ 
if and only if  $-\mu$ is in the $L^s$-spectrum of $D$.
\end{lem}

\begin{proof} The proof follows the lines of
\cite[Appendix~A]{ammann.dahl.humbert:11}. In this reference, $f$  is required
to be a reflection at a hyperplane of $M$. But this doesn't change the part we
need:
We lift $f$ to the bundle $P_{\SO(m)}M$ of oriented orthonormal frames by mapping the frame $\mathcal{E}=(e_1,\ldots,
e_m)$ to $f_*\mathcal{E}=(-\d f(e_1),\d f(e_2),\ldots, \d f(e_m))$, so $f_*\colon 
P_{\SO(m)}M\to P_{\SO(m)}M$.
Since $f$ is an orientation reserving isometry, 
\begin{equation*} f_*(\mathcal{E}A)=f_*(\mathcal{E})JAJ \ \text{for\ all\ } A\in \SO(m)
\end{equation*}
where $J={\rm diag}(-1, 1,1,\ldots, 1)$. The map~$f$ is assumed to lift 
to the spin structure, i.e.,  there is a lift $\tilde{f}_* \colon  P_{\Spin(m)}(M)\to
P_{\Spin(m)}(M)$ with $\theta\circ \tilde{f}_*=f_*\circ \theta$ where $\theta$
denotes the double covering $\theta\colon  P_{\Spin(m)}(M)\to P_{\SO(m)}(M)$. By 
\cite[Lemma~A.1 and Lemma~A.4]{ammann.dahl.humbert:11}, 
$f$ then lifts to a map $f_\sharp\colon  \Si_M\to \Si_M$ on the spinor bundle
which fulfils $f_\sharp (D\phi)=-D (f_\sharp \phi)$.
\end{proof}

\begin{example}\label{spec_or_ex}\hfill\\
{\bf (i)\ } Let $M^{n+1}$ be a Riemannian spin  manifold with a spin structure $\theta$ as above. 
Assume that up to isomorphism this is the unique spin structure on $M$. Let $f\colon  M\to M$ be an orientation reversing
isometry. By pulling back the double covering $P_{\Spin}M\to P_{\SO}M$ by $f_*$ we obtain the double covering $f^*\theta: f^*P_{\Spin}M\to P_{\SO}M$. We then turn $f^*P_{\Spin}M$ into a $\Spin (n+1)$-principal bundle by conjugating the action of $\Spin (n+1)$ on $P_{\Spin}M$ with Clifford multiplication with $e_0$. Then $f^*\theta$ is a spin structure
 on $M$. Thus an isomorphism from $\theta$ to $f^*\theta$ yields a map $f_\sharp$ as above.\\
{\bf (ii)} Consider the map $f=f_1\times \id\colon  \Mc^m=\mH_c^{k+1}\times N^n\to
\Mc^{m,k}$ where $f_1$ is an orientation reversing isometry as in {\rm (i)}. Then, $f$ is again an
orientation reversing isometry. Using $P_{\rm SO}(\mH_c\times N)= (P_{\rm
SO}(\mH_c^{k+1})\times P_{\rm SO}(N))\times_{\bar{\xi}} {\rm SO}(m)$ where
$\bar{\xi}\colon  {\rm SO}(k+1)\times {\rm SO}(n)\to {\rm SO}(m)$ is the standard
embedding and using the analogous describtion for $P_{\Spin}(\mH_c\times N)$, see
Section~\ref{spin_prod}, one see that also $f$ lifts to the spin structure.
\end{example}

\section{Dirac eigenvalues of generic metrics}

\begin{prop}\label{app_C}
  Let $(M,g)$ be a closed, connected Riemannian spin manifold, let $\mu\in \mR$. Let  $U\subset M$ be a nonempty 
 open subset. In case that $\mu=0$, assume additionally that the $\alpha$-genus of $M$ is zero. Then, there is a metric $\tilde{g}$ on $M$ with $\tilde{g}=g$ on $M\setminus U$ and $\ker\, (D^{\tilde{g}}-\mu)=\{0\}$.
\end{prop}

\begin{proof}
 For $\mu=0$, the proposition follows from \cite[Theorem~1.1]{ammann.dahl.humbert:11}. For $\mu\neq 0$, the proof is a direct consequence of the following lemma.
\end{proof}

\begin{lem}
 Let $(M,g)$ be a closed, connected Riemannian spin manifold, let $\mu\in \mR\setminus \{0\}$, and let  $U\subset M$ be a nonempty 
 open subset. Then there is a function $f\in C^\infty(M, 
\mR^+)$ with $f|_{M\setminus U}\equiv 1$ such that $\ker\, (D^{fg}-\mu)=\{0\}$. 
\end{lem}

\begin{proof}
 Choose $f\in C^\infty(M, \mR^+)$ with $f|_{M\setminus U}\equiv 1$ such that $d=\dim (E_{f,\mu}:=\ker\, (D^{fg}-\mu))$ is minimal. Assume $d>0$, and set $g_0=fg$. 
 For $\alpha\in C^\infty(M)$ with $\supp\, \alpha\subset U$ and $t$ close to $0$ we define $g_t:=(1+t\alpha)fg$. Then by \cite{BG_92} there are real analytic functions $\mu_1, \ldots, \mu_d\colon  (-\ep, \ep)\to \mR$ with $\mu_i(0)=\mu$ such that $\spec_{L^2}^M(D^{g_t})\cap (\mu-\delta, \mu+\delta)=\{\mu_1(t), \ldots, \mu_d(t)\}$ including multiplicities. It is shown in \cite{BG_92} that there is an orthonormal basis $(\psi^{(1)}, \ldots , \psi^{(d)})$ of $E_{f,\mu}$ depending on the choice of $\alpha$ such that 
 \[ \frac{\d}{\d t}|_{t=0} \mu_i(t)=-\frac{1}{2} \int_M \< \alpha g_0, Q_{\psi^{(i)}}\> \vo_{g_0}\]
 where $Q_{\psi}(X,Y)=\frac{1}{2}\Re \< X\cdot \nabla_Y \psi + Y\cdot \nabla_X \psi, \psi\>.$ Thus, 
 \[\< g_0, Q_{\psi^{(i)}}\>= \sum_r \< e_r\cdot \nabla_{e_r} \psi^{(i)}, \psi^{(i)}\>= \mu |\psi^{(i)}|^2.\]
 As $d$ is minimal, we see that $ \frac{\d}{\d t}|_{t=0} \mu_i(t)=0$, and thus for all $\alpha$ as above 
 \[ -\frac{1}{2}\int_M \alpha\mu \sum_{i=1}^d |\psi^{(i)}|^2 \vo_{g_0}=0.\]
 Note that  $\phi:=\sum_{i=1}^d |\psi^{(i)}|^2\in C^\infty(M)$ does not depend on the choice of $\alpha$. This can be seen by direct calculation with base change matrices or alternatively by observing that 
 $\phi$ is the pointwise trace of the integral kernel of the projection to $E_{f,\mu}$.
 With $\mu\neq 0$  this implies that $\phi$ and thus all $\psi^{(i)}$ vanish on $U$. The unique continuation principle implies then $\psi^{(i)}\equiv 0$ which gives a contradiction.
\end{proof}

\end{appendix}

%%%%%%%%%%%%%%%%%%%%%%%%%%%%%%%%%%%%%%%%%%%%%%%%%%%%
%\printindex[ref]
%\mnote{Reed-Simons ordnen }
\bibliographystyle{acm}
\bibliography{ammann-grosse}

%%%%%%%%%%%%%%%%%%%%%%%%%%%%%%%%%%%%%%%%%%%%%%%%%%%%%%%%%%%%%%%%%
\end{document}